\documentclass[11pt, reqno]{amsart}

\usepackage{amscd,amssymb, amsmath,amsfonts, wasysym, mathrsfs, enumitem, stmaryrd, mathtools,hhline,color,enumitem, tikz-cd}
\usepackage{graphicx}
\usepackage[all, cmtip]{xy}

\usepackage{url}

\definecolor{hot}{RGB}{65,105,225}

\usepackage[colorlinks=true, linkcolor=hot ,  citecolor=hot, urlcolor=hot]{hyperref}
\usepackage[backend=biber,
style=alphabetic,
url=false,
doi=false,
backref=true
]{biblatex}
\theoremstyle{definition}
\newtheorem{theorem}{Theorem}[section]
\newtheorem{definition}[theorem]{Definition}

\newtheorem{lemma}[theorem]{Lemma}
\newtheorem{proposition}[theorem]{Proposition}
\newtheorem{corollary}[theorem]{Corollary}

\newtheorem{question}[theorem]{Question}
\newtheorem*{theorem*}{Theorem}

\theoremstyle{remark}
\newtheorem{remark}[theorem]{Remark}
\newtheorem{example}[theorem]{Example}

\newcommand\cA{{\mathcal A}}

\renewcommand{\P}{\mathbb{P}}

\def\cA{\mathcal{A}}
\def\Q{\mathbb{Q}}
\def\C{\mathbb{C}}

\def\R{\mathbb{R}}

\def\Z{\mathbb{Z}}

\newcommand{\Trop}{\trop}
                  
\DeclareMathOperator{\Gr}{Gr}
\DeclareMathOperator{\rat}{rat}
\DeclareMathOperator{\MHM}{MHM}
\DeclareMathOperator{\HM}{HM}

\DeclareMathOperator{\Blowup}{Bl}

\def\cM{\mathcal{M}}

\def\cC{\mathcal{C}}

\def\cX{\mathcal{X}}

\def\cT{\mathcal{T}}

\def\cS{\mathcal{S}}

\def\trop{\mathrm{trop}}

\title{Topology and Euler characteristics of tropical varieties}

\author{Scott Hiatt}
\address{University of Wisconsin-Madison} 
\email{shiatt@wisc.edu}

\author{Connor Simpson}
\address{Sydney Mathematics Research Institute \& University of Sydney}
\email{connorgs@connorgs.net}

\author{Botong Wang}
\address{University of Wisconsin-Madison} 
\email{wang@math.wisc.edu}

\author{Chenxi Wu}
\address{University of Wisconsin-Madison} 
\email{cwu367@math.wisc.edu}

\begin{document} 

\begin{abstract}
We study Euler characteristics of tropical subvarieties of tropical abelian varieties. We prove that every H-regular subvariety, locally modeled on tropicalizations of sufficiently well-behaved very affine varieties, has nonnegative signed Euler characteristic. This gives a tropical analogue of a theorem of Green--Lazarsfeld for subvarieties of complex abelian varieties. The main input is a local vanishing theorem for H-regular tropical fans, which also yields a Lefschetz-type theorem for affine H-regular tropical varieties. We further show that the signed Euler characteristic inequality fails for general tropical subvarieties of tropical abelian varieties, and we construct a 3-dimensional tropical fan whose link is not homotopy equivalent to a bouquet of 2-spheres.
\end{abstract}

\maketitle

\section{Introduction}
Let $X$ be a smooth subvariety of a complex abelian variety. A theorem of Green and Lazarsfeld (\cite{GL}) states that the Euler characteristic of $\mathcal{O}_X$ satisfies
\begin{equation*}
(-1)^{\dim X} \chi(X, \mathcal{O}_X)\geq 0.
\end{equation*}

In this paper, we seek an analog of the above inequality for subvarieties of tropical abelian varieties.
A tropical abelian variety is a real torus with a positive definite quadratic form. Our main result concerns H-regular subvarieties, which are locally isomorphic to tropicalizations of certain nice subvarieties of the affine torus (see Definitions \ref{defn:H-regular1} and \ref{defn:H-regular2}).
We prove that the analog of the theorem of Green--Lazarsfeld holds for such tropical subvarieties.

\begin{theorem}\label{thm:abelian variety}
    Let $\cA$ be a tropical abelian variety (see Definition \ref{defn:tropicalAV}) and let $\cX$ be a subvariety of $\cA$ of pure dimension $d$. If $\cX$ is H-regular, then
    \[
    (-1)^d\chi(\cX)\geq 0.
    \]
\end{theorem}
By \cite[Example 17]{TropicalHomology}, $\chi(\cX)=\sum_{q}(-1)^q h_{0, q}(\cX)$, where $h_{0, q}(\cX)$ are the tropical homology Hodge numbers. When $\cX$ is defined by a nice degeneration of subvarieties of abelian varieties, our theorem is a consequence of \cite{GL} and \cite[Theorem 1]{TropicalHomology}. 

The proof of Theorem \ref{thm:abelian variety} (Section \ref{sec:morseargument}) is motivated by the approach in \cite{LMW} and uses Morse theory.
The Morse-theoretic argument rests upon the following local property of H-regular tropical varieties.

\begin{theorem}\label{thm1.2}(Theorem \ref{thm:local})
    Let $\cT$ be a $d$-dimensional H-regular tropical fan in $\R^n$. Let $\ell: \R^n\to \R$ be a linear function that does not vanish on any of the rays of $\cT$. Let $\cT_{\leq 0}=\cT\cap \{\ell\leq 0\}$ be the nonpositive part of $\cT$ and let $L(\cT_{\leq 0})$ be its link. Then the reduced homology of $L(\cT_{\leq 0})$ satisfies
    \[
    \widetilde{H}_k(L(\cT_{\leq 0}); \Q)=0 \quad \text{for all }k\neq d-1. 
    \]
\end{theorem}
To prove Theorem \ref{thm1.2}, we generalize the argument of \cite{Hacking} (see also \cite{Payne}). Instead of studying the mixed Hodge structures on the cohomology groups of an affine variety, we study a quasi-projective variety that admits a proper birational map to an affine variety.

As explained in Example \ref{ex:matroid}, the Bergman fan and the augmented Bergman fan of a matroid realizable over $\C$ are both H-regular. Hence we obtain a different proof of a slightly weaker version of \cite[Theorem 2.1]{AB21}, and thereby answer a question of Mikhalkin--Ziegler for matroids realizable over $\C$ (see \cite{MikhalkinZiegler2008}).

\begin{corollary}\label{cor:Bergman}
    Let $\cT$ be the Bergman fan of a rank $d+1$ matroid realizable over $\C$, or the augmented Bergman fan of a rank $d$ matroid realizable over $\C$. With the notation of Theorem \ref{thm:local}, we have
    \[
    \widetilde{H}_k(L(\cT_{\leq 0}); \Q)=0 \quad \text{for all }k\neq d-1. 
    \]
\end{corollary}

Another consequence of Theorem \ref{thm:local} is the following Lefschetz-type result, comparable to \cite[Theorem 7.4]{AB21}. 

\begin{theorem}\label{thm:Lefschetz}
    Let $\cX\subset \R^n$ be an affine H-regular tropical variety of dimension $d$. Let $H\subset \R^n$ be a general affine hyperplane, that is, an affine hyperplane that is not parallel to any 1-dimensional face of $\cX$. Then 
    \[
    H_k(\cX, \cX\cap H; \Q)=0 \quad \text{for all }k\neq d.
    \]
\end{theorem}

Finally, in Example \ref{ex:inequality fails}, we demonstrate that the inequality of Theorem~\ref{thm:abelian variety} does not hold for arbitrary pure-dimensional subvarieties. In Section~\ref{sec:counterexample}, using similar ideas, we construct a Sch\"on subvariety of $(\C^*)^8$ whose tropicalization has a link that is not homotopy equivalent to a bouquet of spheres. This shows that the question in \cite[Remark~2.11]{Hacking} has a negative answer, even among Sch\"on subvarieties.

\subsection*{Outline of the paper} 
In Section \ref{sec:Hregular}, we introduce H-regular tropical varieties. Section \ref{sec:transversality} develops a notion of transversality in toric varieties and establishes its stability under toric modifications. In Section \ref{sec:weight filtration}, we prove a vanishing theorem for the top weights of cohomology using mixed Hodge modules and the decomposition theorem. Section \ref{sec:dual complex} relates the cohomology of complements of divisors to the homology of their dual complexes.

In Section \ref{sec:localHreg}, we prove the local vanishing theorem for H-regular tropical varieties.
In Section \ref{sec:morseargument}, we deduce the main theorem on the signed Euler characteristic of tropical subvarieties of tropical abelian varieties via a Morse-theoretic argument. In Section \ref{sec:counterexample}, we construct an explicit example of a non-H-regular affine tropical variety and show that it does not exhibit a vanishing property proved by Hacking for H-regular affine tropical varieties. Finally, in Section \ref{sec:questions}, we discuss further questions and possible generalizations of our results.

\subsection*{Acknowledgements}
We are grateful to Kris Shaw for many valuable comments. We also thank Dima Arinkin and Farbod Shokrieh for helpful discussions.

\section{H-regular tropical varieties}\label{sec:Hregular}

The following condition was introduced by Hacking \cite{Hacking}. 

\begin{definition}\label{defn:H-regular1}
Let $U\subset (\mathbb C^*)^n$ be a closed irreducible subvariety. We say that $U$ is \emph{H-regular} if there exists a smooth projective toric variety $X_\Delta$ such that the closure $\overline U$ of $U$ in $X_\Delta$ is smooth and intersects every torus orbit $O\subset X_\Delta$ transversely. Moreover, for every torus orbit \(O\), the intersection $\overline U\cap O$ is either empty or connected.

If such an $X_\Delta$ is fixed, we say that $U$ is H-regular with respect to $X_\Delta$.
\end{definition}

\begin{remark}
    Condition ($*$) in \cite[Theorem 2.5]{Hacking} only requires connectedness when the intersection $\overline{U}\cap O$ is positive dimensional. We need the stronger condition because we will pass to new compactifications by performing weighted blowups along toric boundary strata.
\end{remark}

\begin{definition}\label{defn:H-regular2}
    By a \emph{tropical fan}, we mean a weighted rational polyhedral fan satisfying the balancing condition. In particular, tropicalizations of subvarieties of complex affine tori are tropical fans. 
    
    A tropical fan $\cT\subset \R^n$ is called \emph{H-regular} if $\cT=\trop(U)$ for some H-regular subvariety $U\subset (\C^*)^n$. More generally, a tropical variety $\cX$ is called \emph{H-regular} if, for every vertex $P$ of $\cX$, there is an integral-affine isomorphism from a small neighborhood of $P$ in $\cX$ to a neighborhood of the origin in some H-regular tropical fan.
\end{definition}

\begin{remark}\label{rmk:weight one}
    If $U$ is an H-regular subvariety of $(\mathbb{C}^*)^n$ with respect to a smooth projective toric variety $X_{\Delta}$, then $\trop(U)$ is the union of all cones \(\sigma\in \Delta\) such that $O_\sigma\cap \overline U\neq \emptyset$ \cite[Proposition 6.4.7]{MS}. Moreover, if $\overline U$ meets a torus orbit $O_\sigma\subset X_{\Delta}$ of complementary dimension, then H-regularity implies that $\overline U\cap O_\sigma$ consists of a single point and that the intersection is transverse. It follows that $\sigma$ is a top-dimensional cone of $\trop(U)$ with weight one. Therefore, every top-dimensional cone of an H-regular tropical fan has weight one \cite[Proposition 6.7.2]{MS}.
\end{remark}

\begin{remark}
Because of the connectedness assumption, H-regularity is  strictly stronger than being Sch\"on. Hacking's result is sometimes cited as if it implied that, for every Sch\"on subvariety $U$,
\[
\widetilde{H}_k(L(\trop(U)); \Q)=0 \quad \text{for all } k\neq \dim U-1.
\]
However, we give a counterexample to this statement in Section \ref{sec:counterexample}.
\end{remark}

\begin{example}\label{ex:matroid}
Let $M$ be a loopless matroid realizable over $\mathbb{C}$. Then the \emph{Bergman fan} $\underline{\Pi}_M$ and the \emph{augmented Bergman fan} $\Pi_M$ are both H-regular. 

Indeed, since $M$ is realizable over $\C$, there exists a $\C$-vector space $V$ and an essential hyperplane arrangement $H_1, \dots, H_n\subset V$ whose associated matroid is $M$. Choose an embedding $V\subset \C^n$ such that $H_i$ is the restriction of the $i$-th coordinate hyperplane to $V$. 
Let 
\[
U=\mathbb{P}(V)\cap (\C^*)^{n-1}
\]
and let $\Delta=\Delta_P$ be the $(n-1)$-dimensional permutohedral fan. Then $U$ is H-regular with respect to $X_{\Delta}$, and 
\[
\trop(U)=\underline{\Pi}_M.
\]

Similarly, let $V'\subset \C^n$ be the translation of $V$ by a general vector. Let 
\[
U'=V'\cap (\C^*)^n
\]
and let $\Delta=\Delta_S$ be the $n$-dimensional stellahedron fan. Then $U'$ is H-regular with respect to $X_{\Delta}$, and 
\[
\trop(U')={\Pi}_M.
\]
\end{example}

\begin{example}\label{ex:curve}
Let $\mathcal T$ be a one-dimensional tropical fan in $\mathbb R^2$, and assume that every ray of $\mathcal T$ has weight one. If $\mathcal T$ is not contained in a one-dimensional linear subspace of $\mathbb R^2$, then $\mathcal T$ is H-regular. 

Indeed, choose a complete unimodular fan $\Delta$ in $\mathbb R^2$ that contains $\mathcal T$ as a subfan. The tropical fan $\mathcal T$ determines a divisor class $D$ on the smooth projective toric surface $X_\Delta$. Since the cone of effective divisors of $X_\Delta$ is generated by the torus-invariant divisors, and since $D$ has nonnegative intersection with every torus-invariant divisor, the divisor class $D$ is nef. Hence $D$ is base-point-free.

Moreover, the self-intersection number $D^2$ is equal to the stable self-intersection number of $\cT$, which is strictly positive by the assumption that $\cT$ is not contained in a one-dimensional linear subspace. Hence $D$ is big and base-point-free. It follows, by Bertini's theorem, that a general member of the linear system $|D|$ is irreducible and smooth, and meets the torus orbits transversely. Since all rays of $\cT$ have weight one, each nonempty intersection with a one-dimensional torus orbit consists of a single reduced point.
Thus, if $U$ denotes the restriction of such a general member to the dense torus $(\mathbb C^*)^2$, then $U$ is H-regular with respect to $X_\Delta$, and 
\[
\cT=\trop(U).
\]
\end{example}

Let $\cT$ be a tropical fan in $\R^n$, and let $\ell: \R^n\to \R$ be a linear function. Set 
\[
\cT_{\leq 0}=\cT\cap \{\ell\leq 0\},
\]
and denote the links of $\cT$ and $\cT_{\leq 0}$ at the origin by $L(\cT)$ and $L(\cT_{\leq 0})$, respectively.

\begin{theorem}\label{thm:local}
    Let $\cT$ be a $d$-dimensional H-regular tropical fan.
    If the linear function $\ell$ does not vanish on any of the rays of $\cT$, then
    \[
    \widetilde{H}_k(L(\cT_{\leq 0}); \Q)=0 \quad \text{for all }k\neq d-1. 
    \]
  \end{theorem}
This theorem will be proved in Section \ref{sec:localHreg}.

\begin{definition}\label{defn:tropicalAV}
    A {\em tropical abelian variety} $\cA$ is a real torus $\mathbb{R}^n/\mathbb{Z}^n$ with a positive definite quadratic form on $\mathbb{R}^n$ called the {\em polarization}. A \emph{tropical $d$-cycle} $\cX$ on $\cA$ is a weighted, embedded, rational polyhedral $d$-complex in $\cA$ with integral slopes and satisfying the balancing condition. 
\end{definition}
\begin{remark}
See \cite[Definition 3.3.1]{MS} for the definition of the balancing condition. In this paper, we will not use balancing conditions explicitly. Furthermore, our main results concern tropical varieties whose weights are all one (see Remark \ref{rmk:weight one}). 
\end{remark}

By our main theorem, the sign of the Euler characteristic of an H-regular tropical cycle on a tropical abelian variety is determined by its dimension. Following the idea of \cite{BW}, we construct pure-dimensional subvarieties of a tropical abelian variety that violate the inequality in Theorem \ref{thm:abelian variety}. This shows
that the H-regularity assumption in Theorem~\ref{thm:abelian variety} is necessary. 

\begin{example}\label{ex:inequality fails}
Let $\cC\subset \cA'$ be a theta-type genus 2 curve embedded in its Jacobian (see Figure \ref{fig:curve}, reproduced from \cite[Figure 11]{MZ}). After translating $\cC$, we may assume that one of its vertices is the origin
and that $\cC$ contains no nonzero 2-torsion points of $\cA'$. Let 
\[
\widetilde{\cA}= \cA'\times \cA'\quad \text{and}\quad \widetilde\cX=\cC\times \cC\subset \widetilde{\cA}.
\]
Let $\Gamma$ be the subgroup of \(2\)-torsion points of $\cA'$. Then 
\[
\Gamma \cong \Z/2\Z \times \Z/2\Z.
\]
Consider the diagonal action of $\Gamma$ on $\widetilde{\cA} = \cA' \times \cA'$. 
Let 
\[
\cA=\widetilde{\cA}/\Gamma
\]
and let 
\[
\pi: \widetilde{\cA} \to \cA
\]
be the  quotient map. Set $\cX=\pi(\widetilde{\cX})$. Then $\cX$ is a 2-cycle in $\cA$. We will show that $\chi(\cX)<0$.

Assume the curve $\cC$ in $\cA'$ is represented in Figure \ref{fig:curve}. Translation by one nonzero 2-torsion point $\sigma_1\in A'$ exchanges
$A$ and $B$, that is,
\[
\sigma_1+B=A,\quad\text{and}\quad \sigma_1+A=B.
\]
The other two nonzero 2-torsion points $\sigma_2, \sigma_3$ similarly exchange the pairs $(A, C)$ and $(B, C)$, respectively. Then, 
\[
\widetilde\cX\cap \big(\sigma_1+ \widetilde\cX\big)=\{(A, A), (A, B), (B, A), (B, B)\}
\]
and analogous formulas hold for $\widetilde\cX\cap (\sigma_2+ \widetilde\cX)$  and $\widetilde\cX\cap (\sigma_3+ \widetilde\cX)$. 

The image $\cX=\pi(\widetilde\cX)$ is the quotient of $\widetilde\cX$ by the equivalence relation $x\sim y$ generated by 
\[
y=\sigma_i+x,\qquad i=1, 2, 3.
\]
By the preceding discussion, apart from the reflexive relations, the only nontrivial identifications are
\[
(A, A)\sim (B, B) \sim (C, C), \;\; (A, B)\sim (B, A), \;\; (A, C)\sim (C, A), \;\; (B, C)\sim (C, B).
\]
Therefore, we have
\[
\chi(\cX)=\chi(\widetilde\cX)-5=\chi(\cC)^2-5=1-5=-4.
\]

\begin{figure}[h]
  \centering
  \begin{tikzpicture}
      \draw[thick](0, 0)--(-3, 0);
      \draw[thick](0, 0)--(1, 1);
      \draw[thick](0, 0)--(0, -2);
      \draw[thick](1, 1)--(4, 1);
      \draw[thick](1, 1)--(1, 3);
      \draw[thick](0, -2)--(-1, -3);
      \draw[thick](0, -2)--(3, -2);
      \draw[thick](-1, -3)--(-4, -3);
      \draw[thick](-4, -3)--(-4, -1);
      \draw[thick](-4, -1)--(-3, 0);
      \draw[thick](-3, 0)--(-3, 2);
      \draw[thick](-4, -3)--(-5, -4);
      \draw[thick](4, 1)--(4.3, 1.3);
      \draw[thick](4, 1)--(4, 0.5);
      \draw[thick](1, 3)--(1.3, 3.3);
      \draw[thick](1, 3)--(0.5, 3);
      \draw[thick](3, -2)--(3.3, -1.7);
      \draw[thick](3, -2)--(3, -2.5);
      \draw[thick](-1, -3)--(-1, -3.5);
      \draw[thick](-5, -4)--(-5.5, -4);
      \draw[thick](-5, -4)--(-5, -4.5);
      \draw[thick](-4, -1)--(-4.5, -1);
      \draw[thick](-3, 2)--(-2.7, 2.3);
      \draw[thick](-3, 2)--(-3.5, 2);
      \draw[dashed](4.2, 1.6)--(2.8, -2.6);
      \draw[dashed](1.2, 3.6)--(-1.2, -3.6);
      \draw[dashed](-2.8, 2.6)--(-5.2, -4.6);
      \draw[dashed](-5.4, -4.1)--(3.4, -1.9);
      \draw[dashed](-4.4, -1.1)--(4.4, 1.1);
      \draw[dashed](-3.4, 1.9)--(1.4, 3.1);
      \node at (0.5, -0.1) {$(0, 0)$};
      \node at (0.3, -1){$A$};
      \node at (-0.3, -2.7){$B$};
      \node at (1.5, -1.7){$C$};
      \draw [fill] (0,0) circle [radius=2pt];
      \draw [fill] (0,-1) circle [radius=2pt];
      \draw [fill] (-0.5, -2.5) circle [radius=2pt];
      \draw [fill] (1.5, -2) circle [radius=2pt];
  \end{tikzpicture}
  \caption{A genus two curve embedded in its Jacobian.}
  \label{fig:curve}
\end{figure}
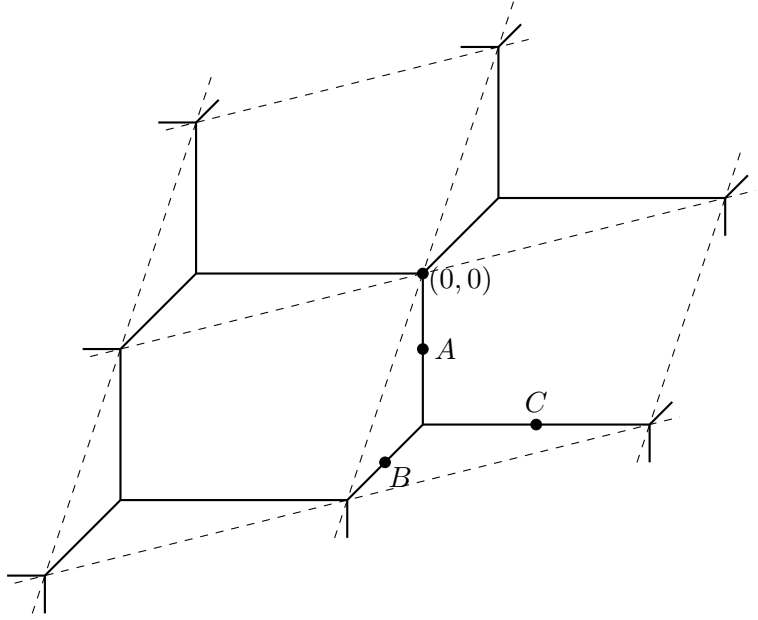

\end{example}

\section{Transversality in toric varieties}\label{sec:transversality}

Let $\sigma\subset \R^n$ be a strongly convex rational polyhedral cone of codimension $m = n - \dim\sigma$. Then $X_\sigma$ is an affine normal toric variety with dense torus $T=(\C^*)^n$. The minimal-dimensional $T$-orbit of $X_\sigma$, denoted $O$, is isomorphic to $(\C^*)^m$. After choosing a noncanonical splitting of the ambient lattice, we may write
\[
X_\sigma\cong (\C^*)^m\times X_{\sigma'},
\]
where $\sigma'$ denotes the same cone $\sigma$, viewed as a
full-dimensional cone in $\operatorname{span}(\sigma)$. The minimal torus orbit of $X_{\sigma'}$ is a point, which we denote by $P$.

\begin{definition}\label{defn:local transversal}
    With the notation above, let $Y\subset X_\sigma$ be a locally closed analytic subvariety of pure codimension $k$. We say that $Y$ intersects $O$ \emph{transversely} if, for any point $x\in Y\cap O$, there exist an analytic neighborhood $U_x$ of $x$ in $X_\sigma$, an analytic neighborhood $U_P$ of $P$ in $X_{\sigma'}$, and an isomorphism of analytic varieties
\begin{equation}\label{eq:analytic iso}
U_x\simeq B^m\times U_P
\end{equation}
such that:
\begin{enumerate}
\item $B^m$ is a complex ball of dimension $m$;
\item under the isomorphism \eqref{eq:analytic iso}, we have
\[
Y\cap U_x\simeq B^{m-k}\times U_P
\]
for a closed $(m-k)$-dimensional complex ball $B^{m-k}\subset B^m$;
\item the product stratification induced by the torus-orbit stratification
of $X_{\sigma'}$ agrees, via \eqref{eq:analytic iso}, with the restriction of the torus-orbit
stratification of $X_\sigma$ to $U_x$.
\end{enumerate}
\end{definition}
\begin{remark}
By definition, if $Y$ intersects $O$ transversely and
$Y\cap O\neq\emptyset$, then $k\leq m$. Equivalently, if
$Y\subset X_\sigma$ is a locally closed analytic subvariety of codimension larger than $m$, then $Y$ intersects $O$ transversely if and only if $Y\cap O=\emptyset$.
\end{remark}
\begin{definition}\label{defn:global transversal}
Let $X$ be a normal toric variety with dense torus $T=(\mathbb C^*)^n$. For every torus orbit $O\subset X$, let $X_{\geq O}$ be the unique affine open toric subvariety of $X$ whose minimal orbit is $O$. Let
$Y\subset X$ be a closed analytic subvariety. We say that $Y$ intersects all torus orbits of $X$ \emph{transversely} if, for every torus orbit $O\subset X$, the subvariety $Y\cap X_{\geq O}$ intersects $O$ transversely in $X_{\geq O}$ in the sense of Definition~\ref{defn:local transversal}.
\end{definition}
\begin{remark}
When $O$ is the dense torus orbit, the above transversality condition requires that $Y\cap O$ is smooth. More generally, if $Y$ intersects all torus orbits of $X$ transversely, then $Y\cap O$ is a closed complex submanifold of $O$ for every torus orbit $O\subset X$.
\end{remark}

\begin{lemma}\label{lemma:two transversal}
Let $X$ be a smooth toric variety. Then an equidimensional closed analytic subvariety $Y\subset X$ intersects all torus orbits transversely if and only if $Y$ is smooth and intersects every torus orbit $O\subset X$ transversely in the usual sense, that is, for every $x\in Y\cap O$,
\[
T_xY+T_xO=T_xX.
\]
\end{lemma}
\begin{proof}
Since both conditions are local on $X$, we may assume that $X$ is a smooth affine toric variety,
\[
X=(\mathbb C^*)^m\times \mathbb C^{n-m}.
\]
It suffices to check the equivalence near a point $x\in Y\cap O$, where
\[
O=(\mathbb C^*)^m\times \{0\}\subset (\mathbb C^*)^m\times \mathbb C^{n-m}.
\]
In this case both $U_x$ and $U_P$ are smooth. Hence the condition that $Y\cap U_x$ is equal to $B^{m-k}\times U_P$ under the product decomposition $U_x\simeq B^m\times U_P$ implies that $Y$ intersects $O$ transversely in the usual sense.

Conversely, assume that $Y$ intersects $O$ transversely in the usual sense. We need to show that $Y$ intersects $O$ transversely in the sense of Definition~\ref{defn:local transversal}. Let $z_1,\ldots,z_{n-m}$ be the coordinates of $\mathbb C^{n-m}$, so that the divisors $z_i=0$ are the toric divisors of $\mathbb C^{n-m}$. We regard the $z_i$ as holomorphic functions on $X=(\mathbb C^*)^m\times \mathbb C^{n-m}$. Since $Y$ is smooth, for every point $x\in Y\cap O$, there exist local holomorphic functions $y_1,\ldots,y_k$ near $x$ such that, schematically,
\[
Y=\{y_1=\cdots=y_k=0\}
\]
near $x$. Since $Y$ intersects $O$ transversely in the usual sense, the cotangent vectors
\[
dz_1,\ldots,dz_{n-m},dy_1,\ldots,dy_k
\]
are linearly independent at $x$. Therefore, by the implicit function theorem, we can find local holomorphic functions
$y_{k+1},\ldots,y_m$ near $x$ such that
\[
y_1,\ldots,y_m,z_1,\ldots,z_{n-m}
\]
form local coordinates of $X$ at $x$. These coordinates define a product decomposition
\[
U_x\simeq B^m\times U_P.
\]
In $B^m$, choose
\[
B^{m-k}=\{y_1=\cdots=y_k=0\}.
\]
Then the two conditions in Definition~\ref{defn:local transversal} are satisfied. Therefore $Y$ intersects $O$ transversely in the sense of Definition~\ref{defn:local transversal}.
\end{proof}

\begin{lemma}\label{lemma:blowup}
Let $X$ be a normal toric variety, and let $Y\subset X$ be an irreducible closed analytic subvariety that intersects every torus orbit transversely in the sense of Definition~\ref{defn:global transversal}. Let
\[
\pi:\widetilde X\to X
\]
be a weighted blowup along a toric subvariety of $X$. Then
$\pi^{-1}(Y)$ is irreducible
and intersects every torus orbit of $\widetilde X$ transversely.
\end{lemma}
\begin{proof}
The statement is local on $X$. To prove the claim near a point
$\widetilde x\in \pi^{-1}(Y)$, we may replace $X$ by the smallest affine open toric subvariety containing $x\coloneqq\pi(\widetilde x)$. Thus, we may assume that $X$ is affine and that $x$ lies in the minimal torus orbit of $X$. With the notation of Definition~\ref{defn:local transversal}, we have
\[
X=X_\sigma\simeq (\mathbb C^*)^m\times X_{\sigma'},
\]
where $X_{\sigma'}$ is an affine toric variety whose minimal torus orbit is a point $P$.

By Definition~\ref{defn:local transversal}, there exists a neighborhood $U_x$ of $x$ in $X$ such that
\[
U_x\simeq B^m\times U_P,
\]
where $U_P$ is a neighborhood of $P$ in $X_{\sigma'}$, and such that this isomorphism restricts to
\[
Y\cap U_x\simeq B^{m-k}\times U_P.
\]
Moreover, for any closed toric subvariety $Z\subset X$, the same product decomposition gives
\[
Z\cap U_x\simeq B^m\times (U_P\cap Z'),
\]
where $Z'\subset X_{\sigma'}$ is a closed toric subvariety.

Both claims now follow from the fact that toric weighted blowups commute with products with a smooth factor. More precisely, the isomorphism $U_x\simeq B^m\times U_P$ induces
\[
\Blowup_Z X\cap \pi^{-1}(U_x) \simeq \Blowup_{Z\cap U_x}U_x \simeq B^m\times \Blowup_{Z'\cap U_P}U_P,
\]
where the weighted blowups are taken with the corresponding toric weights.
Thus,
\begin{equation*}
\pi^{-1}(Y\cap U_x) \simeq B^{m-k}\times \Blowup_{Z'\cap U_P}U_P.
\end{equation*}
This shows that $\pi^{-1}(Y)$ is locally irreducible. Since $\pi$ is an isomorphism over the dense torus $T$, we regard $Y\cap T$ as an open subset of $\pi^{-1}(Y)$. The local product description shows that every neighborhood of a point of $\pi^{-1}(Y)$ meets $Y\cap T$. Hence
$Y\cap T$ is dense in $\pi^{-1}(Y)$. Since $Y\cap T$ is irreducible and $\pi^{-1}(Y)$ is locally irreducible, it follows that $\pi^{-1}(Y)$ is irreducible.

Finally, the same local product description gives the product structures
required in Definition~\ref{defn:local transversal} for the torus orbits of $\widetilde X$. Hence
$\pi^{-1}(Y)$ intersects every torus orbit of $\widetilde X$ transversely.
\end{proof}

\begin{lemma} \label{lemma:connected-blowup} 
With the notation of Lemma~\ref{lemma:blowup}, assume in addition that the intersection of $Y$ with every torus orbit of $X$ is either empty or connected. Then the intersection of $\pi^{-1}(Y)$ with every torus orbit of $\widetilde X$ is either empty or connected. 
\end{lemma}
\begin{proof}
    Given any torus orbit $O$ of $\widetilde X$, its image $O'\coloneqq \pi(O)$ is a torus orbit of $X$. Moreover, the restriction $\pi|_{O}: O\to O'$ is a trivial affine torus bundle. In particular, $\pi|_{O}$ has connected fibers. Hence the lemma follows. 
\end{proof}
\begin{corollary}\label{cor:transversal and rationally smooth}
    Let $\Delta$ be a complete unimodular fan in $\R^n$, and let $X_\Delta$ be the associated complete toric variety with dense torus $T=(\mathbb C^*)^n$. Let $U\subset T$ be an irreducible smooth algebraic subvariety. Assume that the closure $U^\Delta$ of $U$ in $X_\Delta$ is smooth and intersects every torus orbit transversely in the usual sense. Then, for any iterated star subdivision $\widetilde\Delta$ of $\Delta$, the closure $\overline{U}^{\widetilde\Delta}$ of $U$ in $X_{\widetilde\Delta}$ intersects every torus orbit transversely. Moreover, the intersection of $\overline{U}^{\widetilde\Delta}$ with any toric subvariety of $X_{\widetilde\Delta}$ is rationally smooth.
\end{corollary}
\begin{proof}
    Since each star subdivision corresponds locally to a toric weighted blowup, the first statement follows from Lemmas \ref{lemma:two transversal} and \ref{lemma:blowup}. 
    
    For the second statement, denote $\overline{U}^{\widetilde\Delta}$ by $Y$ and let  $Z\subset X_{\widetilde\Delta}$ be a toric subvariety. 
    Since $Y \subset X_{\widetilde\Delta}$ intersects every torus orbit transversely,  for any $x\in Y\cap Z$, there exists a neighborhood $U_x$ of $x$ in $X_{\widetilde\Delta}$ such that
    \[
    U_x\simeq B^{m}\times U_P,
    \]
    where $U_P$ is an open neighborhood in a simplicial affine toric variety $X_{\sigma'}$. Furthermore, this  isomorphism restricts to
    \[
    Y\cap U_x\simeq B^{m-k}\times U_P
    \]
    and
    \[
    Z\cap U_x\simeq B^m\times (U_P\cap Z')
    \]
    where $Z'\subset X_{\sigma'}$ is a toric subvariety. Hence 
    \[
    (Y\cap Z)\cap U_x\simeq B_{m-k}\times (U_P\cap Z').
    \]
    Since $X_{\sigma'}$ is simplicial, the toric variety $Z'$ is also simplicial, hence rationally smooth. It follows that $B^{m-k}\times (U_P\cap Z')$ is rationally smooth. Therefore $Y\cap Z$ is rationally smooth.
\end{proof}

\section{Vanishing in the weight filtration}\label{sec:weight filtration}
In this section, we prove the following vanishing theorem. Although our main interest is the case where 
$X$ is rationally smooth and the vanishing of the top weight piece $Gr_{2d}^{W}$, we state the result in greater generality for independent interest.
\begin{proposition}\label{prop:weight2}
    Let $X$ be an irreducible complex algebraic variety of dimension $d$. Assume there exists a proper birational map $\pi: X \rightarrow Y$ where $Y$ is an affine variety. Then 
    \begin{equation}\label{eq:weight vanishing}
    Gr_{2d}^{W}H^{i}(X, \Q) = 0 \quad \text{and} \quad Gr_{2d-1}^{W}H^{i}(X, \Q) = 0 
    \end{equation}
    for any $i\neq d$.
\end{proposition}

\begin{proof}
    The proof relies on Saito's theory of mixed Hodge modules \cite{Saito2}. An overview of the main definitions and results can be found in \cite{Saito1}. First, for $\cM \in D^{b}\MHM(X)$, we define
    $$p(\cM) = \min\{p: Gr^{F}_{p}DR(\cM) \not \simeq 0\}.$$
    Since the Hodge filtration on $H^{i}(X,DR(\cM))$ is dependent on the Hodge filtration on the de Rham complex $DR(\cM)$, we have
    $$Gr^{F}_{p}H^{i}(X,DR(\cM)) = 0 \quad \text{for $p< p(\cM)$ and $i \in \Z.$}$$
    
    Note that we are using the Hodge filtration as an increasing filtration, in line with Saito's convention, rather than a decreasing filtration. 
    Thus, the $(p,q)$-graded component of the mixed Hodge structure 
    \[
    H^{i}(X,DR(\cM))\cong H^{i}(X, \rat(\cM))\otimes_\Q \C
    \]
    is trivial whenever $p<p(\cM)$ or $q<p(\cM)$. In particular, the $(p, q)$-graded component is trivial if $p+q<2p(\cM)$. With Saito’s convention that the Hodge filtration is increasing, the $(p, q)$-graded component has weight $-(p+q)$. Thus, for $j\geq - 2p(\cM)+1$ and $i \in \Z$,
    \begin{equation}\label{eq: j>-2p}
        Gr^{W}_{j}H^{i}(X, \rat(\cM)) = 0.
    \end{equation}
    
    Since $\Q^{H}_{X}[d] \in D^{b}\MHM(X)$ and $IC^{H}_{X} \in \HM_{X}(X,d)$, by \cite[Corollary~0.3]{saito4} we have
    $$p(\Q^{H}_{X}[d]) = p(IC^{H}_{X}) = -d = -\dim X,$$
    and
    $$ Gr^{F}_{-d}DR(\Q^{H}_{X}[d]) \simeq Gr^{F}_{-d}DR(IC^{H}_{X}).$$
    By \cite[Proposition~1.14]{Saito2}, we have the identity 
    \[
    Gr^{W}_{d}H^{0}(\Q^{H}_{X}[d]) =IC^{H}_{X}.
    \]
    Therefore, there exists $\mathcal{N} \in D^{b}\MHM(X)$ such that we have the following distinguished triangle in $D^{b}\MHM(X),$
\[
\mathcal{N} \longrightarrow  \Q^{H}_{X}[d] \longrightarrow  \displaystyle IC^{H}_{X} \xrightarrow{\; +1\; }.
\]
    This triangle induces a long exact sequence of mixed Hodge structures
    $$
    \cdots \rightarrow  H^{i-d}(X, \rat(\mathcal{N})) \rightarrow H^{i}(X, \Q) \rightarrow  IH^{i}(X, \Q) \rightarrow H^{i +1-d}(X, \rat(\mathcal{N})) \rightarrow \cdots,
    $$
    where the intersection cohomology groups $IH^{i}(X, \Q)=H^{i-d}(X, IC_{X}(\Q))$. Because the map 
    \[
    Gr^{F}_{-d}DR(\Q^{H}_{X}[d]) \rightarrow Gr^{F}_{-d}DR(IC^{H}_{X})
    \]
    is a quasi-isomorphism, we must have 
    \[
    Gr^{F}_{p}DR(\mathcal{N} ) \simeq 0\quad \text{for $p < -d+1$.}
    \]
    This implies
    $$Gr^{W}_{j}H^{i-d}(X, \rat(\mathcal{N})) = 0 \quad \text{for $j \geq 2d-1$ and  $i \in \Z.$}$$
    Thus by the long exact sequence, we have
    \begin{equation}\label{eq:IH iso1}
    Gr^{W}_{j}H^{i}(X, \Q) \cong  Gr^{W}_{j}IH^{i}(X, \Q),
    \end{equation}
    for $j \geq 2d-1$ and $i \in \Z$.
    Therefore, it suffices to prove that 
    \[
    Gr_{2d}^{W}IH^{i}(X, \Q) = 0 \quad \text{and} \quad Gr_{2d-1}^{W}IH^{i}(X, \Q) = 0\quad \text{for $i\neq d$.}
    \]

    Since $\pi:X \rightarrow Y$ is proper and birational, by the decomposition theorem (\cite[Theorem 1.12]{Saito2}), we have
    \[
    \pi_{+}IC^{H}_{X} \simeq IC_{X}^{Y} \oplus \bigoplus_{\ell\in \Z}\cM_\ell[-\ell],
    \]
    where each $\cM_\ell$ is a pure Hodge module whose support is a proper closed subset of $Y$. Furthermore, $\cM_\ell=0$ for all but finitely many $\ell$. 
    Taking cohomology, we have
    \begin{align*}
    IH^{i}(X, \Q) &\cong IH^{i}(Y, \Q) \oplus \bigoplus_{\ell \in \Z} H^{i-d}\big(Y, \rat(\cM_\ell[-\ell])\big)\\
    &\cong IH^{i}(Y, \Q) \oplus \bigoplus_{\ell \in \Z} H^{i-d-\ell}\big(Y, \rat(\cM_\ell)\big).
    \end{align*}
    By \cite[Proposition 2.6]{saito3}, we have 
    \[
    Gr^{F}_{p}DR(\cM_\ell) = 0 \quad \text{for $p < -d+1$ and $\ell \in \Z$.}
    \]
    Thus, 
    \[
    p(\cM_\ell)\geq -d+1\quad \text{for $\ell\in \Z$.}
    \]
    By \eqref{eq: j>-2p},
    \[
    Gr^{W}_{j}H^{i}(Y,\rat(\cM_\ell)) = 0 \quad \text{for $j\geq 2d-1$ and $i,\ell \in \Z$.}
    \]
    Therefore, 
    \begin{equation}\label{eq:IH iso2}
    Gr^{W}_{j}IH^{i}(X, \Q) \cong Gr^{W}_{j}IH^{i}(Y, \Q), \quad \text{for  $j \geq 2d-1$ and $i \in \Z$.}
    \end{equation}

    Since $Y$ is affine, by Artin's vanishing theorem (see \cite[Th\'eor\`eme~4.1.1]{BBD}), we have
    \[
    IH^{i}(Y, \Q) = 0\quad \text{for $i > d$.}
    \]
    By the following lemma, we have
    \[
    Gr^{W}_{j}IH^{i}(Y, \Q)=0 \quad\text{for $i< d$ and $j\geq 2d-1$.}
    \]
    Therefore, we have
    \[
    Gr_{2d}^{W}IH^{i}(Y, \Q) = 0 \quad \text{and} \quad Gr_{2d-1}^{W}IH^{i}(Y, \Q) = 0\quad \text{for $i\neq d$.}
    \]
    Thus, by \eqref{eq:IH iso1} and \eqref{eq:IH iso2}, we obtain the desired vanishing result \eqref{eq:weight vanishing}.
\end{proof}

\begin{lemma}
    Let $Y$ be an irreducible complex quasi-projective variety. Then,
    \[
    Gr^{W}_{j}IH^{i}(Y, \Q) = 0\quad \text{for $j>2i$.}
    \]
\end{lemma}
\begin{proof}
    Let $\widetilde{Y}\to Y$ be a resolution of singularities. Then, by the decomposition theorem, the Hodge structure $IH^{i}(Y, \Q)$ is a direct summand of $H^i(\widetilde{Y}, \Q)$. Since $\widetilde{Y}$ is a smooth quasi-projective variety, it follows from \cite[Corollaire 3.2.15]{Deligne} that
    \[
    Gr^{W}_{j}H^{i}(\widetilde{Y}, \Q) = 0\quad \text{for $j>2i$.}
    \]
    Since $IH^{i}(Y, \Q)$ is a direct summand of $H^i(\widetilde{Y}, \Q)$ as Hodge structures, the same vanishing holds for the weight graded components of $IH^{i}(Y, \Q)$. 
\end{proof}

\section{The cohomology of the dual complex of divisors}\label{sec:dual complex}
The following theorem is a generalization of \cite[Theorem 3.1]{Hacking}, where the projective variety is smooth and the divisors have simple normal crossings. 
In the rationally smooth setting, instead of using differential forms as in loc. cit., we use constructible sheaves to deduce that the Gysin sequence computes the cohomology of a rationally smooth quasi-projective variety. For more details about the Gysin sequence, see \cite[Example 5.3]{DP}.
\begin{proposition}\label{prop:dual complex}
    Let $X$ be a rationally smooth projective variety of dimension $d$. Let $D_1, \dots, D_l$ be a collection of divisors on $X$ such that for any $I\subset \{1, \dots, l\}$, $D_I\coloneqq \bigcap_{j\in I}D_j$ is connected, rationally smooth and of codimension $|I|$ (when $|I|>d$, $D_I$ is empty). Let $K$ be the dual complex of $D=D_1\cup \cdots \cup D_l$. Denote $X\setminus D$ by $U$. Then, 
\[
\widetilde{H}_i(K; \Q)\cong \mathrm{Gr}_{2d}^W H^{2d-(i+1)}(U, \Q).
\]
\end{proposition}
\begin{proof}
For $1\leq k\leq d$, let $D^{(k)}=\bigcup_{|I|=k}D_I$, and let $\widetilde{D}^{(k)}$ be the normalization of $D^{(k)}$. Denote by $\iota_k: \widetilde{D}^{(k)}\to X$ the composition of the normalization map $\widetilde{D}^{(k)}\to {D}^{(k)}$ and the closed embedding $D^{(k)}\to X$. 
    Let $j: U\to X$ be the open embedding. Then, we have a \v{C}ech-type exact sequence
    \[
    0\to j_!\Q_{U}\to \Q_X\to \iota_{1*}(\Q_{\widetilde{D}^{(1)}})\to \iota_{2*}(\Q_{\widetilde{D}^{(2)}})\to \cdots \to \iota_{d*}(\Q_{\widetilde{D}^{(d)}})\to 0.
    \]
    Then the hypercohomology spectral sequence satisfies 
    \[
    E_1^{pq}=H^q\big(X, \iota_{p*}(\Q_{\widetilde{D}^{(p)}})\big)\cong \bigoplus_{|I|=p}H^q(D_I, \Q)\Rightarrow H^{p+q}(X, j_!\Q_{U})\cong H^{p+q}_c(U, \Q),
    \]
    where we set $\widetilde{D}^{(0)}=D_\emptyset=X$, and $\iota_0: X\to X$ is the identity map. Moreover, all the arrows in the spectral sequence are maps of Hodge structures. Hence, it degenerates at the $E_2$-page (see \cite[Lemme (3.2.10)]{Deligne}). 
    Denote the $p$-th cohomology of the complex 
    \[
    0\to H^q(X, \Q)\to H^q(\widetilde{D}^{(1)}, \Q)\to \cdots \to H^q(\widetilde{D}^{(d-1)})\to H^q(\widetilde{D}^{(d)}, \Q)\to 0
    \]
    by $N^{p,q}$, where $H^q(X, \Q)$ is in degree 0. Then, the degeneration of the above spectral sequence at the $E_2$-page implies that 
    \[
    H_c^{k}(U, \Q)\cong \bigoplus_{p+q=k}N^{p,q}.
    \]
    Since $H^q(D_I, \Q)$ is a pure Hodge structure of weight $q$ and 
    \[
    H^q(\widetilde{D}^{(i)}, \Q)\cong \bigoplus_{|I|=i}H^q({D}_{I}, \Q),
    \]
    $N^{p,q}$ is a pure Hodge structure of weight $q$. Hence, 
    \[
    \mathrm{Gr}_{2d-m}^W H_c^{2d-k}(U, \Q)\cong N^{m-k, 2d-m}.
    \]
    By Poincar\'e duality, 
    \[
    H^{2d-k}(U, \Q)\cong H^{k}_c(U, \Q)^\vee.
    \]
    Since the Poincar\'e duality is compatible with mixed Hodge structures, we have
    \[
    \mathrm{Gr}_{2d-m}^W H^{2d-k}(U, \Q)\cong \big(\mathrm{Gr}_{m}^W H^{k}_c(U, \Q)\big)^\vee.
    \]
    Therefore, 
    \[
    \mathrm{Gr}_{2d}^W H^{2d-(i+1)}(U, \Q)\cong \big(\mathrm{Gr}_{0}^W H^{(i+1)}_c(U, \Q)\big)^\vee\cong \big(N^{i+1, 0}\big)^\vee.
    \]
    By definition, $\big(N^{i+1,0}\big)^\vee$ is isomorphic to the $(i+1)$-th homology of the complex
    \[
    0\to H^0(\widetilde{D}^{(d)}, \Q)^\vee\to H^0(\widetilde{D}^{(d-1)}, \Q)^\vee\to \cdots \to H^0(\widetilde{D}^{(1)}, \Q)^\vee\to H^0(X, \Q)^\vee\to 0
    \]
    where $H^0(X, \Q)^\vee$ is in degree 0. The above complex is isomorphic to 
    \begin{multline*}
    0\to \bigoplus_{|I|=d}H^0({D}_I, \Q)^\vee\to \bigoplus_{|I|=d-1} H^0({D}_I, \Q)^\vee\to \cdots \\
    \to \bigoplus_{|I|=1} H^0({D}_I, \Q)^\vee\to H^0(X, \Q)^\vee\to 0.
    \end{multline*}
    Since $H^0({D}_I, \Q)^\vee\cong \Q$ for all $I$ with $|I|\leq d$, the $(i+1)$-th homology of the above complex is equal to the $i$-th reduced homology of the dual complex $K$. 
\end{proof}
Combining Propositions \ref{prop:weight2} and \ref{prop:dual complex}, we immediately have the following corollary.
\begin{corollary}\label{cor:dual complex vanishing}
    Let $X$, $D$ and $U$ be as in Proposition \ref{prop:dual complex}. Assume further that $U$ admits a proper birational map to an affine variety. Then, the dual complex $K$ of $D$ satisfies
    \[
    \widetilde{H}_i(K; \Q)=0
    \]
    unless $i=d-1$.
\end{corollary}

\section{Vanishing result for \texorpdfstring{$H$}{H}-regular tropical varieties}
\label{sec:localHreg}
In this section, we prove Theorem \ref{thm:local}. Given an H-regular tropical variety $\cT$, by definition it is the tropicalization of an H-regular subvariety $U\subset (\C^*)^n$. By taking closures in appropriately chosen toric varieties, we construct (partial) compactifications, 
\[
U\subset U^+\subset \overline{U},
\]
where $\overline{U}$ is projective, and $U^+$ admits a proper birational map to an affine variety. 
We then show that $L(\cT_{\leq 0})$ is homotopy equivalent to the dual complex of the divisor $\overline{U} \setminus U^+$. Theorem~\ref{thm:local} then follows from Corollary~\ref{cor:dual complex vanishing}.
The following result will help us to construct appropriate (partial) compactifications of $U$.
\begin{proposition}\label{prop:Delta plus}
  Let $\Delta$ be a complete simplicial fan in $\R^n$, and let $\ell: \R^n \to \R$ be a linear functional that is non-constant on all rays of $\Delta$.
  Let $\Sigma \subset \Delta$ be a subfan with support $\cS := |\Sigma|$, and let $\Sigma^\leq$ be the subfan of $\Sigma$ consisting of cones on which $\ell$ is non-positive.

  There exists a rational simplicial cone $\sigma_\cS \subset \{\ell > 0\} \cup \{0\}$ such that $L(\cS \setminus \sigma_\cS)$ deformation retracts onto both $L(\cS \cap \{\ell \leq 0\})$ and $L(\Sigma^\leq)$.
      Moreover, there is a simplicial refinement $\Delta'$ of $\Delta$, obtained by a sequence of stellar subdivisions, with a subfan $\Delta^+ \subset \Delta'$ such that $|\Delta^+| = \sigma_\cS$. 
\end{proposition}
Our proof of Proposition~\ref{prop:Delta plus} is inspired by \cite[Proof of Lemma~4.3]{CCKR}.
\begin{proof}
  \renewcommand{\cT}{\cS} 
    \newcommand{\conv}{\operatorname{conv}}
  Let $(\;,\,)$ be the standard inner product on $\R^n$, and let $v \in \R^n$ be the unit vector such that $(v, -)$ is a positive multiple of $\ell$.
  Choose rational unit vectors $v_1, \ldots, v_n \in \R^n$ such that the cone 
  \[
  \sigma_\cS := \{ u \in \R^n : (u, v_i) \geq 0, 1 \leq i \leq n\}
  \]
  contains $v$ in its relative interior. By choosing $v_1, \ldots, v_n$ sufficiently close to $v$, we may further assume that $\sigma_\cS \setminus \{0\}$ is contained in $\{\ell > 0\}$ and that any ray of $\Delta$ that lies in $\{\ell > 0\}$ is in the relative interior of $\sigma_\cS$.

  We now show that $L(\cT \setminus \sigma_\cS)$ deformation retracts as claimed.
  Let $\cS_\leq := \cS \cap \{\ell \leq 0\}$.
  Write $\conv(A)$ for the convex hull of a set $A \subset \R^n$, and for each ray $\rho$ of $\Delta$, let $u_\rho \in \R^n$ be the unit vector generating $\rho$.
    For a cone $\sigma \in \Delta$, define
    \begin{align*}
      C^-(\sigma) &:= \conv(u_\rho : \ell(u_\rho) < 0 \text{ and $\rho$ is a ray of  $\sigma$}) \\
      C^+(\sigma) &:= \conv(u_\rho : \ell(u_\rho) > 0 \text{ and $\rho$ is a ray of $\sigma$}) \\
      C(\sigma) &:= \conv(u_\rho : \text{$\rho$ is a ray of $\sigma$})
    \end{align*}
    Observe that $L(\Delta)$ is homeomorphic to $\cup_{\sigma \in \Delta} C(\sigma)$, and that restricting this homeomorphism yields homeomorphisms of $L(\cT \setminus \sigma_\cS)$, $L(\cT_{\leq 0})$, and $L(\Sigma^\leq)$ with
    \begin{align*}
      C(\cT \setminus \sigma_\cS) &:= \cup_{\sigma \in \Sigma} (C(\sigma) \setminus \sigma_\cS), \\
      C(\cT_{\leq0}) &:= \cup_{\sigma \in \Sigma} (C(\sigma) \cap \{\ell \leq 0\}), \text{ and} \\
      C(\Sigma^\leq) &:= \cup_{\sigma \in \Sigma^\leq} C(\sigma),
    \end{align*}
    respectively. Accordingly, it suffices to show that $C(\cT\setminus \sigma_\cS)$ deformation retracts onto both $C(\cT_{\leq0})$ and $C(\Sigma^\leq)$. We will accomplish this by defining homotopies pieceswise on cones of $\Sigma$.

  Let $\sigma \in \Sigma$ be a cone not contained in $\sigma_\cT$.
  If $\sigma\subset \{\ell < 0\}$, then define $H_\sigma: C(\sigma) \times [0, 1] \to C(\sigma)$ by $H_\sigma^\cS(x,s) = x$, and set $H_\sigma^\Sigma := H_\sigma^\cS$.

  Otherwise, $\sigma$ has rays in both $\{\ell > 0\}$ and $\{\ell < 0\}$.
  Recall that the \emph{join} of $C^-(\sigma)$ and $C^+(\sigma)$ is
  \begin{align*}
    C^-(\sigma) * C^+(\sigma) &:= C^-(\sigma) \times C^+(\sigma) \times [0, 1] / \sim, \text{ where } \\
                        & (x, y, 0) \sim (x, y', 0) \text{ and } (x, y, 1) \sim (x', y, 1).
  \end{align*}
  Since $\sigma$ is simplicial, we may identify $C^-(\sigma) * C^+(\sigma)$ with $C(\sigma)$ via the homeomorphism
  \[
    h_\sigma: C^-(\sigma) * C^+(\sigma) \to C(\sigma), \quad (x,y,t) \mapsto (1-t)x + ty.
  \]
  For each $(x,y) \in C^-(\sigma) \times C^+(\sigma)$ there is a unique $m_{xy} \in (0, 1)$ such that $\ell(h_\sigma(x,y,m_{xy})) = 0$. Moreover, $m_{xy}$ is continuous in $x$ and $y$.
  Define maps
  \begin{align*}
    H_\sigma^\cS: (C(\sigma) \setminus C^+(\sigma)) \times [0, 1] &\to C(\sigma) \setminus C^+(\sigma) \\
    (h_\sigma(x,y,t),s) &\mapsto h_\sigma\left(x,y, \min\{t, st + (1-s) m_{xy}\}\right) \\
    H_\sigma^\Sigma: (C(\sigma) \setminus C^+(\sigma)) \times [0, 1] &\to C(\sigma) \setminus C^+(\sigma) \\
    (h_\sigma(x,y,t),s) &\mapsto h_\sigma\left(x,y, st\right).
  \end{align*}
  We note several properties of $H_\sigma^\cS$ and $H_\sigma^\Sigma$.
  \begin{itemize}
      \item The restrictions $H_\sigma^\cS(-, 1)$ and $H_\sigma^\Sigma(-, 1)$ are the identity.
      \item If $h_\sigma(x,y,t) \not \in \sigma_\cS$, then for all $s \in [0, 1]$, neither $H^\cS_\sigma(h_\sigma(x,y,t), s)$ nor $H^\Sigma_\sigma(h_\sigma(x,y,t),s)$ lies in $\sigma_\cS$.
    \item If $h_\sigma(x,y,t) \in \{\ell \leq 0\}$, then for all $s \in [0,1]$,  
    \[
    H^\cS_\sigma\big(h_\sigma(x,y,t),s\big) = h_\sigma(x,y,t).
    \]
    \item If $\tau$ is a face of $\sigma$ not contained in $\sigma_\cS$, then the restrictions of $H_\sigma^\cS$ and $H_\sigma^\Sigma$ to $C(\tau) \setminus C^+(\tau) \times [0,1]$ are equal to $H^\cS_\tau$ and $H_\tau^\Sigma$, respectively.
  \end{itemize}
  These properties imply that the maps $\{H_\sigma^\cS : \sigma \in \Sigma, \sigma \not \subset \sigma_\cS\}$ and $\{H_\sigma^\Sigma : \sigma \in \Sigma, \sigma \not \subset \sigma_\cS\}$ glue, and that the glued maps restrict to deformation retractions
  \[
    H^\cS, H^\Sigma: C(\cT \setminus \sigma_\cS) \times [0, 1] \to C(\cT \setminus \sigma_\cS),
  \]
  of $C(\cT \setminus \sigma_\cS)$ onto $C(\cT_{\leq 0})$ and $C(\Sigma^\leq)$, respectively.

  To complete the proof, it remains to establish the ``moreover'' statement.
  Let $\Delta''$ be any complete rational simplicial fan in which $\sigma_\cT$ is a cone.
  By \cite[Theorem 1.1]{AP24}, $\Delta$ and $\Delta''$ have a common rational simplicial refinement $\Delta'$, obtainable from both $\Delta$ and $\Delta''$ by a sequence of stellar subdivisions. We may take $\Delta^+$ to be the subfan of $\Delta'$ consisting of all cones contained in $\sigma_\cT$.
\end{proof}

Given a $d$-dimensional H-regular tropical fan $\cT$ in $\R^n$, there exists an H-regular subvariety $U\subset (\C^*)^n$ such that $\cT=\trop(U)$. By the definition of H-regular variety, there is a complete unimodular fan $\Delta$ such that the closure $\overline{U}$ of $U$ in $X_{\Delta}$ satisfies the conditions of Definition \ref{defn:H-regular1}. Let $\Sigma\subset \Delta$ be the subfan consisting of cones $C$ whose corresponding orbit $O_C$ intersects $\overline{U}$. By \cite[Proposition 3.9]{ST} and Lemma \ref{lemma:connected-blowup}, we have 
\[
|\Sigma|=\trop(U)=\cT.
\]


Let $\Delta'$ be an iterated weighted star subdivision of $\Delta$ and $\Delta^+\subset \Delta'$ be a subfan satisfying the properties of Proposition \ref{prop:Delta plus} with $\cS=\cT$. Denote by $U^+$ the closure of $U$ in $X_{\Delta^+}$. 

\begin{lemma}
    Under the above notations, $X_{\Delta^+}$ admits a proper map to an affine toric variety, which restricts to an isomorphism on the open torus $(\C^*)^n$.
\end{lemma}
\begin{proof}
  The support $|\Delta^+|$ is a convex rational polyhedral cone, with associated affine toric variety $X_{|\Delta^+|}$.
  The identity map $\R^n \to \R^n$ induces a proper toric map $X_{\Delta^+} \to X_{|\Delta^+|}$, which is an isomorphism on dense tori (see, e.g. \cite[Theorem 3.4.11]{CLS11}).
\end{proof}

\begin{corollary}\label{cor:U^+ admits map to affine}
    The partial compactification $U^+$ of $U$ admits a proper birational map to an affine variety. 
\end{corollary}
\begin{proof}
    Consider the composition $U^+\to X_{\Delta^+}\to X_{|\Delta^+|}$. The first map is a closed embedding, and the second is a proper map. Thus, $U^+$ admits a proper map to an affine variety. Moreover, since $X_{\Delta^+}\to X_{|\Delta^+|}$ restricts to an isomorphism between the open tori, and since $U^+$ intersects the open torus of $X_{\Delta^+}$ nontrivially, the composition induces a birational map from $U^+$ to its image, which is an affine variety.
\end{proof}
Let $K$ be the dual complex of the divisor $\overline{U}\setminus U^+$. Then, by Corollary \ref{cor:dual complex vanishing}, we have
\begin{equation}\label{eq:vanish reduced coh of K}
    \widetilde{H}_i(K; \Q)=0
\end{equation}
unless $i=d -1$.

\begin{proposition}\label{prop:K and T}
Under the above notations, we have a homotopy equivalence
\[
K\simeq L(\cT\setminus \sigma_\cT)\simeq L(\cT_{\leq 0}),
\]
where $\sigma_\cT$ is a rational simplicial cone associated to $\cT$ as in Proposition \ref{prop:Delta plus}.
\end{proposition}
\begin{proof}
Since our construction satisfies  Proposition \ref{prop:Delta plus}, the inclusion 
\[
L(\cT_{\leq 0})\subset L(\cT\setminus \sigma_\cT)
\]
is a homotopy equivalence. Hence, we have proved the second homotopy equivalence.

Let $\Sigma'$ be the subfan of $\Delta'$ whose support is equal to $\cT$. Then, $\Sigma'$ is a refinement of $\Sigma$. Denote by $\Sigma^c$ the subfan of $\Sigma'$ generated by all the rays not contained in $\sigma_\cT$. First, we notice that the rays of $\Sigma^c$ correspond to the irreducible components of the boundary divisor $\overline U \setminus U^+$. 
It follows from \cite[Proposition 3.9]{ST} that as simplicial complexes $K\cong L(|\Sigma^c|)$
where as the link of a simplicial fan, $L(|\Sigma^c|)$ is naturally a simplicial complex. 

Furthermore, notice that there is a strong deformation retract from $L(|\Sigma'|\setminus \sigma_\cT)$ to $L(|\Sigma^c|)$. In fact, both $L(|\Sigma'|\cap \sigma_\cT)$ and $L(\Sigma^c)$ are simplicial subcomplexes of $L(|\Sigma'|)$, and they are generated by vertices that form a partition of the vertices of $L(\Delta)$. Given any simplex $\sigma$ and a face $\tau$ of $\sigma$, there is a standard linear retract of $\sigma$ to $\tau$. Using such a retract, we can construct a deformation retract from $L(|\Sigma'|)\setminus L(|\Sigma'|\cap \sigma_\cT)=L(|\Sigma'|\setminus \sigma_\cT)$ to $L(\Sigma^c)$. Therefore, we have
\[
K\simeq L(|\Sigma^c|)\simeq L(|\Sigma'|\setminus \sigma_\cT)=L(\cT\setminus \sigma_\cT).
\]
Thus, the first homotopy equivalence follows. 
\end{proof}
Finally, we can deduce Theorem ~\ref{thm:local}. 
\begin{proof}[Proof of Theorem ~\ref{thm:local}]
    The desired statement follows from Proposition \ref{prop:K and T} and the vanishing property \eqref{eq:vanish reduced coh of K}.
\end{proof}

\section{Morse theory arguments}\label{sec:morseargument}
We now prove Theorem \ref{thm:abelian variety} via an argument similar to the one in \cite{LMW} (see also \cite{EGM}). 

Let $\cX$ be a $d$-dimensional tropical subvariety of an $n$-dimensional tropical abelian variety $\cA$. For every vertex $v$ in $\cX$, we say that a linear function $\ell$ defined on some neighborhood of $v$ is {\em Morse} if it is not constant on any $1$-simplex containing $v$. Given any Morse linear function $\ell$ near a vertex $v$, let
$L_v$ be the link of $\cX$ at $v$, and let 
\[
L^{\leq}_{\ell, v}=L_v\cap \{x\in U_v\mid \ell(x)\leq \ell(v)\},
\]
where $U_v$ is a neighborhood of $v$ which contains $L_v$ and on which $\ell$ is defined. Theorem \ref{thm:abelian variety} will be deduced from the following theorem. 
\begin{theorem}\label{thm:top} 
Under the above notations, suppose that for all vertices $v$ of $\cX$, and any local Morse linear function $\ell$ near $v$, 
\begin{equation}\label{eq_assumption}
\widetilde{H}^i\big(L^{\leq}_{\ell, v}; \Q\big)=0\quad \text{for all} \;\; i\neq d-1,
\end{equation}
or more generally, 
\begin{equation*}
(-1)^{d-1}\big(\chi(L^{\leq}_{\ell, v})-1\big)\geq 0.
\end{equation*}
Then we have
    \[
    (-1)^d\chi(\cX)\geq 0.
    \]
\end{theorem}

\begin{remark}
When $d=1$, the assumption of Theorem \ref{thm:top} is trivially true, and the conclusion is valid as well because
if $G$ is a finite graph with positive number of edges and without leaf, then $\chi(G)\leq 0$.
\end{remark}


Let $f: X\to S^1=\R/\Z$ be a circle-valued function from a topological space $X$. Given any $a\in \R$ and interval $[a, b]\subset \R$ with $a\leq b< a+1$, by abusing notation, we also use them to denote their images in $S^1$, and we denote $f^{-1}(a)$ and $f^{-1}([a, b])$ by $X_a$ and $X_{[a, b]}$, respectively. The following is a basic fact in circle-valued Morse theory. We include the proof for the reader's convenience. 
\begin{lemma}\label{lemma_chi}
    Let $X$ be a finite CW complex. Let $f: X\to S^1=\R/\Z$ be a circle-valued function such that the preimage of any point or interval has finite-dimensional cohomology over $\Q$. Assume that away from a finite set $W\subset S^1$, $f$ is a fiber bundle. 
    Then,
    \[
    \chi(X)=\sum_{c\in W}\chi\big(X_{[c-\epsilon, c+\epsilon]}, X_{c-\epsilon}\big)
    \]
    for sufficiently small $\epsilon>0$, where the relative Euler characteristic is defined by
    \[
    \chi\big(X_{[c-\epsilon, c+\epsilon]}, X_{c-\epsilon}\big)=\sum_{i\geq 0}(-1)^i\dim H^i\big(X_{[c-\epsilon, c+\epsilon]}, X_{c-\epsilon}; \Q\big).
    \]
\end{lemma}
\begin{proof}
If $W=\emptyset$, then $X$ is a fiber bundle over $S^1$. The Euler characteristic of a fiber bundle is equal to the product of the Euler characteristics of the fiber and the base. Hence, $\chi(X)=0$, and the desired equation holds. For the rest of the proof, we assume that $W\neq \emptyset$. 

    Fix $b\in S^1\setminus W$, and sufficiently small $\epsilon>0$. Then, by the long exact sequence of relative cohomology, we have
    \begin{equation}\label{eq_fb1}
    \chi(X)=\chi\big(X, X_{[b-\epsilon, b+\epsilon]}\big)+\chi\big(X_{[b-\epsilon, b+\epsilon]}\big).
    \end{equation}
    By excision and the long exact sequence of relative cohomology, we have 
    \begin{equation}\label{eq_fb2}
    \begin{aligned}
    \chi\big(X, X_{[b-\epsilon, b+\epsilon]}\big)&=\chi\big(X_{[b+\epsilon, b+1-\epsilon]}, X_{b+\epsilon}\cup X_{b+1-\epsilon}\big)\\
    &=\chi\big(X_{[b+\epsilon, b+1-\epsilon]}, X_{b+\epsilon}\big)-\chi\big(X_{b+1-\epsilon}\big).
    \end{aligned}
    \end{equation}
Since $f$ is a fiber bundle over $[b-\epsilon, b+\epsilon]$, we have
\[
X_{[b-\epsilon, b+\epsilon]}\simeq X_{b-\epsilon}=X_{b+1-\epsilon}. 
\]
Thus, combining equations \eqref{eq_fb1} and \eqref{eq_fb2}, we have
\begin{equation}\label{eq_fb3}
\chi(X)=\chi\big(X_{[b+\epsilon, b+1-\epsilon]}, X_{b+\epsilon}\big).
\end{equation}
    
    By excision and the long exact sequences of relative cohomology, for any $a_1, a_2, a_3$ with $a_1<a_2<a_3<a_1+1$, we have
\begin{equation*}
\begin{aligned}
\chi\big(X_{[a_1, a_3]}, X_{a_1}\big)&=\chi\big(X_{[a_1, a_3]}, X_{[a_1, a_2]}\big)+\chi\big(X_{[a_1, a_2]}, X_{a_1}\big) \\
  &= \chi\big(X_{[a_2, a_3]}, X_{a_2}\big)+\chi\big(X_{[a_1, a_2]}, X_{a_1}\big).
\end{aligned}
\end{equation*}
Therefore, breaking $[b+\epsilon, b+1-\epsilon]$ into smaller closed intervals containing exactly one point of $W$ each, and using the fact that $f$ is a fiber bundle away from $W$, we have
\begin{equation}\label{eq_fb4}
\chi\big(X_{[b+\epsilon, b+1-\epsilon]}, X_{b+\epsilon}\big)=\sum_{c\in W}\chi\big(X_{[c-\epsilon, c+\epsilon]}, X_{c-\epsilon}\big).
\end{equation}
Now, the desired equation follows from \eqref{eq_fb3} and \eqref{eq_fb4}.
\end{proof}


\begin{proof}[Proof of Theorem \ref{thm:top}]
Since the statement is topological, we may identify the tropical abelian variety $\cA$ with $\mathbb{R}^n/\mathbb{Z}^n$. Given a $d$-dimensional tropical cycle $\cX$ of $\cA$, we may choose a general real Lie group map $f_\cA:\cA\to S^1=\R/\Z$ such that all vertices of $\cX$ have distinct images. In particular, any local lift of $f_\cA$ is a Morse function. Denote the restriction by $f=f_\cA|_{\cX}: \cX\to S^1$. 

Let $V\subset \cX$ be the set of vertices, and let $W=f(V)$. Then, $f$ is obviously a fiber bundle over $S^1\setminus W$. Given $c\in W$, we consider the pair $(\cX_{[c-\epsilon, c+\epsilon]}, \cX_{c-\epsilon})$ for sufficiently small $\epsilon>0$. 
By our assumption, there is a unique $v\in V$ such that $f(v)\in [c-\epsilon, c+\epsilon]$. Notice that for $1\gg \delta\gg \epsilon>0$, the inclusions
\[
\cX_{c-\epsilon}\subset \big(\cX_{[c-\epsilon, c+\epsilon]}\setminus B^\circ_\delta(v)\big)\cup \cX_{c-\epsilon}
\]
and
\[
\cX_{[c-\epsilon, c+\epsilon]}\cap B_\delta(v)\subset \cX\cap B_\delta(v)
\]
are both deformation retracts, where $B_\delta(v)\subset \cA$ is the ball centered at $v$ with radius $\delta$ and $B_\delta^\circ(v)$ is its interior. Therefore, using excision we can deduce that for any $i\geq 0$,
\begin{align*}
    &H^i\big(\cX_{[c-\epsilon, c+\epsilon]}, \cX_{c-\epsilon};\Q\big)\\
    \cong & H^i\big(\cX_{[c-\epsilon, c+\epsilon]}, \big(\cX_{[c-\epsilon, c+\epsilon]}\setminus B_\delta^\circ(v)\big)\cup \cX_{c-\epsilon};\Q\big)\\
    \cong &H^i\big(\cX_{[c-\epsilon, c+\epsilon]}\cap B_\delta(v), \big(\cX_{[c-\epsilon, c+\epsilon]}\big)\cap \partial B_\delta(v)\big)\cup \big(\cX_{c-\epsilon}\cap B_\delta(v)\big);\Q\big)\\
    \cong &H^i\big(\cX\cap B_\delta(v), \cX_{[c-0.5, c+\epsilon]}\cap \partial B_\delta(v); \Q\big)
\end{align*}
where $\partial B_\delta(v)$ is the boundary of $B_\delta(v)$. Therefore, by the long exact sequence of relative cohomology,
\begin{align*}
\chi\big(\cX_{[c-\epsilon, c+\epsilon]}, \cX_{c-\epsilon}\big)&=\chi\big(\cX\cap B_\delta(v), \cX_{[c-0.5, c+\epsilon]}\cap \partial B_\delta(v)\big)\\
&=\chi(\cX\cap B_\delta(v))-\chi\big(\cX_{[c-0.5, c+\epsilon]}\cap \partial B_\delta(v)\big).
\end{align*}
Notice that $\cX\cap B_\delta(v)$ is contractible and $\cX_{[c-0.5, c+\epsilon]}\cap \partial B_\delta(v)$ is homotopy equivalent to $\cX_{[c-0.5, c]}\cap \partial B_\delta(v)$, which can be regarded as the half link $L^{\leq}_{\ell, v}$. Therefore, 
\[
\chi\big(\cX_{[c-\epsilon, c+\epsilon]}, \cX_{c-\epsilon}\big)=1-\chi(L^{\leq}_{\ell, v}).
\]
Thus, by \eqref{eq_assumption}, we have
\[
(-1)^d\chi\big(\cX_{[c-\epsilon, c+\epsilon]}, \cX_{c-\epsilon}\big)\geq 0
\]
for any $c\in W$. Now, we can apply Lemma \ref{lemma_chi} and conclude that
\[
(-1)^{d}\chi(\cX)\geq 0. \qedhere
\]
\end{proof}

\begin{proof}[Proof of Theorem \ref{thm:abelian variety}]
    By Theorem \ref{thm:local}, the vanishing assumption \eqref{eq_assumption} in Theorem \ref{thm:top} holds for H-regular tropical varieties. Therefore, by Theorem \ref{thm:top}, we conclude that
    \[
    (-1)^d\chi(\cX)\geq 0
    \]
    for any $d$-dimensional H-regular tropical subvariety of a tropical abelian variety.
\end{proof}
Theorem \ref{thm:Lefschetz} can be proved using similar arguments.
\begin{proof}[Proof of Theorem \ref{thm:Lefschetz}]
    Let $\ell: \R^n\to \R$ be a linear function such that the restriction of $\ell$ to any one-dimensional face of $\cX$ is non-constant, and assume that $H=\{\ell=c\}$ for some constant $c\in \R$. For any $t\in \R$ and any closed interval $[a, b]$, set
    \[
    \mathcal X_t=\mathcal X\cap \{\ell=t\}\quad \text{and}\quad\mathcal X_{[a,b]}=\mathcal X\cap \{a\leq \ell\leq b\}. 
    \]

    Assume that $\{\ell=b\}$ contains exactly one vertex $v$ of $\cX$. Denote the link of $\cX$ at $v$ by $L_v$ and let 
\[
L^{\leq}_{\ell, v}=L_v\cap \{x\in U_v\mid \ell(x)\leq \ell(v)\}.
\]
Then, for any $a<b$ and $1\gg \delta\gg \epsilon>0$, we have
\begin{align*}
H_k(\cX_{[a,\, b]}, \cX_{[a,\, b-\epsilon]};\Q)&\cong H_k(\cX_{[a,\, b]}\cap B_\delta(v), \cX_{[a,\, b-\epsilon]}\cap B_\delta(v);\Q)\\
&\cong H_k(\cX\cap B_\delta(v), \cX_{[a,\, b-\epsilon]}\cap \partial B_\delta(v);\Q)\\
&\cong \widetilde{H}_{k-1}(\cX_{[a,\, b-\epsilon]}\cap \partial B_\delta(v); \Q)\\
&\cong \widetilde{H}_{k-1}(L^{\leq}_{\ell, v_i}; \Q),
\end{align*}
where the first isomorphism follows from the same arguments as the proof of Theorem \ref{thm:top}, the third isomorphism follows from $\cX\cap B_\delta(v)$ being contractible. 

In general, if $\{\ell=b\}$ contains more than one vertex of $\cX$, then the above formula can be replaced by
\[
H_k(\cX_{[a,\, b]}, \cX_{[a,\, b-\epsilon]};\Q)\cong \bigoplus_{i}\widetilde{H}_{k-1}(L^{\leq}_{\ell, v_i}; \Q),
\]
where the sum is over all vertices $v_i$ of $\cX$ in $\{\ell=b\}$. 

In any case, it follows from the H-regular assumption of $\cX$ and Theorem~\ref{thm:local} that 
\[
\widetilde{H}_{k-1}(L^{\leq}_{\ell, v_i}; \Q)=0 \quad \text{for all }k\neq d.
\]
Hence for any $a<b$ and sufficiently small $\epsilon>0$, 
\[
H_k(\cX_{[a,\, b]}, \cX_{[a,\, b-\epsilon]};\Q)=0\quad \text{for all }k\neq d.
\]

Now, let $a_1, \dots, a_\mu, b_1, \dots, b_\nu$ be the values of $\ell$ on the set of vertices of $\cX$ that are different from $c$ such that 
\[
a_\mu<\cdots a_1<c<b_1<\cdots <b_\nu.
\]
Then, for sufficiently small $\epsilon>0$, the inclusion $\cX_{c}\subset \cX_{[c,\, b_1-\epsilon]}$ is a homotopy equivalence. In particular,
\[
H_k(\cX_{[c,\, b_{1}]}, \cX_{c};\Q)\cong H_k(\cX_{[c,\, b_{1}]}, \cX_{[c,\, b_{1}-\epsilon]};\Q)
\]
which vanishes except when $k=d$.

Similarly, for any $1\leq i\leq \nu-1$, the inclusion $\cX_{[c,\, b_i]}\subset \cX_{[c,\, b_{i+1}-\epsilon]}$ is also a homotopy equivalence. Hence,
\[
H_k(\cX_{[c,\, b_{i+1}]}, \cX_{[c,\, b_{i}]};\Q)\cong H_k(\cX_{[c,\, b_{i+1}]}, \cX_{[c,\, b_{i+1}-\epsilon]};\Q)
\]
which also vanishes except when $k=d$.

Consider the long exact sequence on relative homology groups
\begin{multline*}
    \cdots \to H_k(\cX_{[c,\, b_{i}]}, \cX_{c};\Q)\to H_k(\cX_{[c,\, b_{i+1}]}, \cX_{c};\Q)\\
    \to H_k(\cX_{[c,\, b_{i+1}]}, \cX_{[c, b_i]};\Q)\to H_{k-1}(\cX_{[c,\, b_{i}]}, \cX_{c};\Q)\to \cdots.
\end{multline*}
We have proved that both $H_k(\cX_{[c,\, b_{1}]}, \cX_{c};\Q)$ and $H_k(\cX_{[c,\, b_{i+1}]}, \cX_{[c,\, b_{i}]};\Q)$ vanish except when $k=d$. Thus, using induction on $i$, we can conclude that for any $1\leq i\leq \nu$
\[
H_k(\cX_{[c,\, b_{i}]}, \cX_{c};\Q)=0 \quad \text{for all }k\neq d. 
\]
In particular,
\begin{equation}\label{eq:rel1}
    H_k(\cX_{[c,\, b_{\nu}]}, \cX_{c};\Q)=0 \quad \text{for all }k\neq d. 
\end{equation}

Replacing $\ell$ by $-\ell$ and applying similar arguments, we can inductively show that 
\[
H_k(\cX_{[a_i,\, b_{\nu}]}, \cX_{[c, b_{\nu}]};\Q)=0 \quad \text{for all }k\neq d. 
\]
In particular,
\begin{equation}\label{eq:rel2}
    H_k(\cX_{[a_\mu,\, b_{\nu}]}, \cX_{[c, b_{\nu}]};\Q)=0 \quad \text{for all }k\neq d. 
\end{equation}
By \eqref{eq:rel1} and \eqref{eq:rel2}, and the long exact sequence of relative homology groups, we conclude that 
\[
H_k(\cX_{[a_\mu,\, b_{\nu}]}, \cX_{c};\Q)=0 \quad \text{for all }k\neq d.
\]
Since $\cX$ has no vertex in $\{\ell>b_\nu\}\cup \{\ell<a_\mu\}$, the inclusion $\cX_{[a_\mu,\, b_{\nu}]}\subset \cX$ is a homotopy equivalence. Therefore, we obtain the desired vanishing
\[
H_k(\cX, \cX_{c};\Q)=0 \quad \text{for all }k\neq d.\qedhere
\]
\end{proof}

\section{A Sch\"on but non-H-regular variety}\label{sec:counterexample}
In this section, we construct a $3$-dimensional sch\"on subvariety
\[
U\subset (\mathbb{C}^*)^8
\]
admitting a tropical compactification with a disconnected positive-dimensional torus-orbit intersection. We then prove that the link $L(\trop(U))$ has nontrivial first homology, which implies that $\trop(U)$, and hence $U$, is not H-regular. Consequently, the connectedness assumption in H-regularity is essential for both Theorem \ref{thm:local} and \cite[Theorem 2.5]{Hacking}. In \cite[Remark~2.11]{Hacking}, Hacking suggested that there should exist tropical fans violating this vanishing property, although no such example was known at the time. Our construction provides such an example.

\smallskip
We first construct the toric ambient space.
Let $e_1, \ldots, e_9 \in \R^9$ be the standard basis of $\R^9$, and let $e_0 = - \sum_{i=1}^9 e_i$.
The fan with cones generated by proper subsets of $\{e_0, \ldots, e_9\}$ has associated toric variety $\P^9$.
Let $\Sigma_0$ be the 3-dimensional fan obtained from the 3-skeleton of this fan by performing stellar subdivision at the cones spanned by $\{e_1, e_2, e_3\}$ and $\{e_4, e_5, e_6\}$.
The rays of $\Sigma_0$ are generated by the vectors $e_0, \ldots, e_9$ and
\[ a = e_1 + e_2 + e_3 \quad\text{and}\quad b = e_4+e_5+e_6.
\]
Its maximal cones have generators
\begin{itemize}
  \item $\{e_i, e_j, e_k\}$ with $0 \leq i < j < k \leq 9$, except for $\{e_1, e_2, e_3\}$ and $\{e_4, e_5, e_6\}$,
  \item $\{a, e_i, e_j\}$ with $1 \leq i < j \leq 3$, and
  \item $\{b, e_i, e_j\}$ with $4 \leq i < j \leq 6$.
\end{itemize}
Write $\rho_a$ and $\rho_b$ for the rays generated by $a$ and $b$.

The surjection of lattices $\Z^9 \to \Z^9 / (a-b) \cong \Z^8$ induces a linear map $Q: \R^9 \to \R^9 / (a-b) \cong \R^8$. Consider the 3-dimensional fan in $\R^8$:
\[\Sigma := \{Q(\sigma) : \sigma \in \Sigma_0\}. \]

\begin{proposition}\label{ambient}
  The collection of cones $\Sigma$ is a fan with $|\Sigma| = Q(|\Sigma_0|)$.
  Every cone in $\Sigma_0$ maps isomorphically onto its image, and distinct cones in $\Sigma_0$ have distinct images, except for $\rho_a$ and $\rho_b$, which are mapped onto the same ray of $\Sigma$.

  Consequently, $Q$ induces a map of toric varieties 
  \[
  q: X_{\Sigma_0} \to X_\Sigma.
  \]
  Each torus orbit $O_\sigma$ of $X_{\Sigma_0}$ is mapped onto a torus orbit $O_{Q(\sigma)}$ in $X_\Sigma$ with 1-dimensional fibers that are the orbits of the 1-parameter subgroup
  \begin{equation}\label{eq_lambda}
  \lambda(t) = (t,t,t,t^{-1},t^{-1},t^{-1},1,1,1)
  \end{equation}
   on $O_\sigma$.
 \end{proposition}
 \begin{proof}
   The toric statements follow from the polyhedral ones by standard results on toric varieties.

   No cone in the 3-skeleton of the fan of $\P^9$
 contains a nonzero vector with exactly three coordinates equal to zero.
   Therefore, the 1-dimensional subspace $(a-b)$ intersects every cone of $\Sigma_0$ only at the origin, proving that $Q$ maps every cone of $\Sigma_0$ isomorphically onto its image.

   The claim about which cones have distinct images is verified by the Macaulay2 code in Appendix \ref{M2cert}.
 \end{proof}

 We now construct a sch\"on subvariety of $X_\Sigma$ that is not $H$-regular.
Let $G \subset \Gr(4, \C^{10})$ be the open set of subspaces realizing the uniform matroid of rank 4 on 10 elements.
If $V \in G$, then the projection 3-plane $\P(V)$ meets all torus orbits of $\P^9$ transversely, hence is contained in the non-compact toric variety obtained from $\P^9$ be removing all torus orbits of codimension greater than 3.

By definition, $\Sigma_0$ is the refinement of a subfan of the fan of $\P^9$. The induced toric variety map
\[
  \pi: X_{\Sigma_0} \to \P^9
\]
is the composition of a blow-down map and an open embedding. The argument in the previous paragraph implies that if $V\in G$, then $\P(V)$ is contained in the image of $\pi$, and we denote by $Y_V$ the proper transform of $\P(V)$ under $\pi$.
Let $x_1, \ldots, x_9$ be coordinates on the toric chart of $\P^9$ corresponding to the cone generated by $\{e_1,\ldots,e_9\}$. 
Then $Y_V$ is isomorphic to the blow-up of $\P(V)$ at two points:  $p_a := \P(V) \cap \{x_1=x_2=x_3=0\}$ and $p_b := \P(V) \cap \{x_4=x_5=x_6 = 0\}$.
\begin{proposition}\label{Yschon}
  For a general choice of $V \in G$, the restriction 
  \[
  q|_{Y_V}: Y_V \to X_\Sigma
  \]
  is a closed immersion. Its image meets every torus orbit of $X_\Sigma$ transversely, hence $U := q(V \cap (\C^*)^9)$ is sch\"on.
  However, there is a torus orbit of $X_\Sigma$ whose intersection with $\overline{U} = q(Y_V)$ is disconnected.
\end{proposition}
For a map $f: A \to B$ of smooth varieties, we denote by $f_x: T_x A \to T_{f(x)} B$ the induced map on the tangent space at $x\in A$.
\begin{lemma}\label{YTY}
  Let $V \in G$, let $x \in X_{\Sigma_0}$ and let $v \in T_x X_{\Sigma_0}$ be a tangent vector.
  \begin{enumerate}
    \item\label{inY} $x \in Y_V$ if and only if $\pi(x) \in \P(V)$.
    \item \label{inTY} If $x \in Y_V$, then $v \in T_x Y_V$ if and only if $\pi_x (v) \in T_{\pi(x)}\P(V)$.
    \item \label{kerq} If $x \in Y_V$ and $v \in T_x X_{\Sigma_0}$ is a nonzero tangent vector to the curve $t \mapsto \lambda(t) \cdot x$, then $\pi_x(v) \neq 0$. Here, the action $\lambda(t)$ is defined as in \eqref{eq_lambda}.
  \end{enumerate}
\end{lemma}
\begin{proof}
Statement (\ref{inY}) is equivalent to the equality 
\[
\pi^{-1}(\P(V))=Y_V.
\]
Since $\P(V)$ intersects all the torus orbits transversely, the above equality follows from Lemma \ref{lemma:blowup}. 

Notice that the argument in Lemma \ref{lemma:blowup} is schematic. More precisely, it implies that $\pi^{-1}(\P(V))$ is reduced and irreducible, and $\pi^{-1}(\P(V))=Y_V$ as subschemes of $X_{\Sigma_0}$. Thus, for $v\in T_xX_{\Sigma_0}$,
\[
v\in T_xY_v\iff Z(v)\subset Y_v \iff \pi(Z(v))\subset \P(V)\iff \pi_x (v) \in T_{\pi(x)}\P(V)
\]
where $Z(v)$ denotes the Artinian subscheme of length 2 corresponding to $v$. 
Therefore, (\ref{inTY}) follows. 

Notice that the action $\C^*$-actions on $\pi(X_{\Sigma_0})$ induced by $\lambda: \C^*\to (\C^*)^9$ has no fixed points. Hence, for any $x\in X_{\Sigma_0}$, the map \[
t\mapsto \pi(\lambda(t)\cdot x)
\]
induces a locally closed embedding $\C^*\to \P^9$. Therefore, in the following commutative diagram, 
\[
\xymatrix{
& T_xY_V\ar@{^{(}->}[d] \ar[r]^{(\pi|_{Y_V})_x\quad}& T_{\pi(x)}\P(V)\ar@{^{(}->}[d]\\
T_1\C^*\ar[r]& T_xX_{\Sigma_0} \ar[r]^{\pi_x}& T_{\pi(x)}\P^9
}
\]
the composition of the bottom row is injective. By definition, $v\in T_xY_V$ is nonzero and (as an element in $T_xX_{\Sigma_0}$) in the image of $T_1\C^*$. Thus, the image of $v$ in $T_{\pi(x)}\P(V)$ is nonzero, and hence (\ref{kerq}) follows.
%
\end{proof}
\begin{lemma}\label{Yinj}
  For a general $V\in G$, $q|_{Y_V}$ is injective.
\end{lemma}
\begin{proof}
  Let $O_\sigma$ and $O_\tau$ be torus orbits in $X_{\Sigma_0}$ that map onto the same torus orbit $O$ in $X_\Sigma$. By Proposition \ref{ambient}, either $\sigma = \tau$, or, after possibly swapping $\sigma$ and $\tau$, we have $\sigma = \rho_a$ and $\tau = \rho_b$. 

  Now consider the correspondence
  \[
    R_{\sigma, \tau} = \{(W, x, y) \in G \times (O_\sigma \times_O O_\tau)\mid x,y \in Y_W, x \neq y \}.
  \]
  We claim that $\pi(x) \neq \pi(y)$ for all $(W, x,y)\in R_{\sigma, \tau}$. 
Indeed, the images $\pi(O_{\rho_a})$ and $\pi(O_{\rho_b})$ are disjoint torus orbits in $\P^9$. Hence, it suffices to treat the case $\sigma=\tau$. In that situation, since $x$ and $y$ have the same image in $O$, they lie in the same orbit of the $\C^*$-action defined by $\lambda: \C^*\to (\C^*)^9$. Any $\C^*$-fixed points on $\P^9$ has at least 6 zero projective coordinates, and no such point lies in the image of $\pi$. Thus, the $\C^*$-action on $\pi(X_{\Sigma_0})$ has no fixed points, and hence the $\C^*$-action on $X_{\Sigma_0}$ has no fixed points either. In particular, the restriction of $\pi$ to any $\C^*$-orbit is injective. Since $x$ and $y$ belong to the same $\C^*$-orbit, $x\neq y$ implies that $\pi(x)\neq \pi(y)$, proving the claim. 

  By the claim and Lemma \ref{YTY}(\ref{inY}), if $(x,y) \in O_\sigma \times_O O_\tau$, then the fiber of $R_{\sigma, \tau}$ over $(x,y) \in O_\sigma \times_O O_\tau$ has dimension at most $\dim \Gr(2, \C^8) = 12$.
  The base $O_\sigma \times_O O_\tau$ has dimension at most $\dim(X_{\Sigma_0})+1 = 10$.
  Hence, 
  \[
  \dim R_{\sigma, \tau} \leq 10 + 12 < 24 = \dim Gr(4, \C^{10}).
  \]

  It follows that the image of the projection $R_{\sigma, \tau} \to \Gr(4, \C^{10})$ is contained in a Zariski closed proper subset. Since there are only finitely many pairs of orbits $O_\sigma, O_\tau$, for a general choice of $W \in \Gr(4, \C^{10})$, the restriction map $q|_{Y_W}$ is injective .
\end{proof}
\begin{lemma}\label{TYinj}
  For a general $V\in G$, $q|_{Y_V}$ is injective on all tangent spaces.
\end{lemma}
\begin{proof}
  Let $O_\sigma$ be a torus orbit in $X_{\Sigma_0}$.
  Let
  \[
    K_\sigma = \{(x,v) \in TX_{\Sigma_0} \mid x \in O_\sigma, \, 0\neq v \in T_x X_{\Sigma_0}, \, q_x(v) =0\}.
  \]
  The projection $K_\sigma\to O_\sigma$ has 1-dimensional fibers. Hence, 
  \[
  \dim K_\sigma\leq \dim X_{\Sigma_0}+1=10. 
  \]
  Now, consider
  \[
    R_\sigma = \{(W, (x, v))\in G \times K_\sigma \mid x \in Y_W,  v \in T_x Y_W\}.
  \]
  By Lemma \ref{YTY} (3), the fiber of $R_\sigma$ over a point $(x,v) \in K_\sigma$ has dimension at most $\dim \Gr(2, \C^8) = 12$.
  Hence, 
  \[
  \dim R_\sigma \leq \dim K_\sigma +12\leq 10 + 12 = 22.
  \]
  It follows that the image of the projection $R_\sigma \to G$ is contained in a Zariski closed proper subset. Since there are finitely many orbits $O_\sigma$, $q|_{Y_W}$ is injective on tangent spaces for a general choice of $W$.
\end{proof}

\begin{proof}[Proof of Proposition \ref{Yschon}]
  By Lemmas \ref{Yinj} and \ref{TYinj}, we can pick a general $V \in G$ such that
  \begin{enumerate}
      \item $\P(V)$ meets all torus orbits of $\P^9$ transversely;
      \item $q|_{Y_V}$ is injective;
      \item $q|_{Y_V}$ is injective on tangent spaces.
  \end{enumerate}
For such $V$, $q|_{Y_V}: Y_V\to X_\Sigma$ is a closed immersion.

Since $q: X_{\Sigma_0}\to X_{\Sigma}$ is a locally a principal $\C^*$-fibration, $q$ is a submersion. For any torus orbit $O_\sigma\subset X_{\Sigma_0}$, a general choice of $V\in G$, and any $x\in Y_V\cap O_\sigma$, 
\begin{equation}\label{3equalities}
    T_{q(x)}q(Y_V)+ T_{q(x)}q(O_\sigma)= q_x(T_xY_V+T_xO_\sigma)= q_x(T_xX_{\Sigma_0})=T_{q(x)}X_{\Sigma},
\end{equation}
where the first equality follows from $q|_{Y_V}$ being an immersion, the second follows from the intersection of $Y_V$ and $O_\sigma$ being transverse, and the last follows from $q$ being a submersion. Since for any torus orbit $O_{\sigma'}\subset X_{\sigma}$ and $x'\in q(Y_V)\cap O_{\sigma'}$, there exist a torus orbit $O_\sigma\subset X_{\Sigma_0}$ and $x\in Y_V\cap O_\sigma$, the equality \eqref{3equalities} implies that $q(Y_V)$ intersects every torus orbit of $X_{\Sigma}$ transversely. 

  Finally, for the statement on disconnectedness, observe that 
  \[
  q(Y_V) \cap q(O_{\rho_a}) = q\big((Y_V \cap O_{\rho_a}) \cup (Y_V \cap O_{\rho_b})\big)
  \]
  is not connected.
\end{proof}

\begin{remark}
  If we apply the same construction with $V$ a general element of $\Gr(4, \C^9)$, the image of the point $\P(V) \cap \{x_0=x_7=x_8=0\}$ in $\overline{U}$ is a singular point.
  In this case, we do not know whether $U$ is sch\"on. In our construction, we do not have this issue because 
  \[
  \P(V) \cap \{x_0=x_7=x_8=x_9=0\}=\emptyset.
  \]
\end{remark}
\begin{proposition}
  The link $L$ of $\Trop(\overline U)$ at the origin is homotopy equivalent to a wedge of eighty-four 2-spheres with a pair of points identified. Consequently,
  \[ H_i(L, \Z) = \begin{cases} \Z, & i=0 \\
    \Z, & i=1 \\
    \Z^{84}, & i=2\\
    0, & i \geq 3
  \end{cases}
\]
\end{proposition}
\begin{proof}
  The set $\trop(V \cap (\C^*)^9)$ is the support of the Bergman fan of a uniform matroid of rank 4 on 10 elements.
  This set contains both $\rho_a$ and $\rho_b$. By \cite[Corollary following Theorem 1]{AK06}, the link of $\Trop(V \cap (\C^*)^9)$ at the origin is homotopy equivalent to a wedge of eighty-four 2-spheres.

  By \cite[Corollary 3.2.13]{MS},
  \[ \Trop(U) = \Trop(q)(\Trop(V \cap (\C^*)^9)) = Q(\trop(V \cap (\C^*)^9)). \]
  Hence, $L$ is homotopy equivalent to a wedge of eighty-four 2-spheres with a pair of points identified, which has the described homology groups.
\end{proof}

\section{Further discussion and questions}\label{sec:questions}
\subsection*{Generalizations to arbitrary matroids}
A geometric proof of a statement about realizable matroids often suggests a purely combinatorial proof, which would extend the statement to arbitrary matroids. This leads to the following question.
\begin{question}
    Can the conclusion of Corollary \ref{cor:Bergman} be proved for arbitrary matroids using combinatorial Hodge theory?
\end{question}

In the current proof of Corollary \ref{cor:Bergman}, we use refinements of (augmented) Bergman fans and the decomposition theorem. Both of these ingredients have combinatorial counterparts, namely \cite{ADH} and \cite{AHL}. However, it is not clear to us how to formulate a combinatorial analogue of Artin vanishing for perverse sheaves on an affine variety.

\subsection*{Euler characteristics of $\Omega^p$}
As discussed after Theorem \ref{thm:abelian variety}, the Euler characteristic of a tropical variety $\mathcal X$ can be expressed in terms of the tropical homology
Hodge numbers introduced in \cite{TropicalHomology}, 
\[
\chi(\cX)=\sum_{q}(-1)^q h_{0, q}(\cX).
\]
On the other hand, the generic vanishing theorem (\cite[Theorem~2]{GL}) implies that for a $d$-dimensional smooth subvariety $X$ of an abelian variety,
\[
(-1)^{d-p}\chi(X, \Omega^p)=\sum_{q} (-1)^{d-p-q}h^{p, q}(X)\geq 0.
\]
This suggests the following tropical analogue.
\begin{question}
Let \(\mathcal X\) be a \(d\)-dimensional H-regular subvariety of a tropical
abelian variety. Do the tropical homology Hodge numbers of \(\mathcal X\) satisfy
\[
(-1)^{d - p}\,\chi_p(\cX)\coloneqq \sum_{q}(-1)^{d-p-q}h_{p, q}(\cX) \geq 0
\]
for all $p$?
\end{question}

\subsection*{Fourier--Mukai transformations}
Recently, Ghosh and Shokrieh \cite{GS} constructed a cohomological Fourier--Mukai transform for tropical abelian varieties. Since the theorem of Green--Lazarsfeld that 
\[
(-1)^{\dim X}\chi(X, \mathcal{O}_X)\geq 0
\]
for smooth closed subvarieties $X$ of complex abelian varieties can be proved using Fourier--Mukai transformations \cite{PP}, it is natural to ask whether there is a tropical counterpart of this approach.

\begin{question}
Can the cohomological Fourier--Mukai transform of \cite{GS} be used to give an alternative proof of Theorem \ref{thm:abelian variety}? More generally, can it be used to identify classes of cycles in tropical abelian varieties with nonnegative signed Euler characteristic?
\end{question}

\subsection*{Non-abelian analogues}
Some non-abelian analogues of the Green--Lazarsfeld theorem, and of its topological counterpart, the Franecki--Kapranov theorem \cite{FK}, have recently been established (see \cite{DW2,AW,DW}). In the classical setting, non-abelian analogues of subvarieties of abelian varieties include
smooth projective varieties with negative sectional curvature, and more generally smooth projective varieties with large fundamental group, such as compact smooth Shimura varieties. It would be interesting to develop a corresponding theory in tropical geometry.
\begin{question}
Is there a natural class of ``negatively curved'' or ``hyperbolic'' tropical varieties, or of tropical varieties with large fundamental group, that is closed under Cartesian products and passage to smooth subvarieties? Moreover, does every H-regular tropical variety in such a class have nonnegative signed Euler characteristic?
\end{question}

\subsection*{Chern classes}
For a compact complex manifold, the topological Euler characteristic is equal to the integral of the top Chern class, or equivalently to the degree of the zero-dimensional homology Chern class. Thus the other Chern classes may be viewed as refinements of the Euler characteristic.

The Chern classes of smooth subvarieties of abelian varieties satisfy signed positivity properties. For example, if $X$ is a smooth subvariety of a complex abelian variety, then 
$\Omega_X^1$ is globally generated.
Consequently, the signed Chern classes of $X$ are effective. More precisely, if $c_k(X)$ denotes the $k$-th cohomological Chern class of $T_X$, then $(-1)^k c_k(X)$ is effective.

The Chern--Schwartz--MacPherson (CSM) classes of tropical manifolds were introduced in \cite{ChernClasses}. Here a tropical manifold is a tropical variety locally isomorphic to Bergman fans of matroids (see \cite[Definition~2.3]{ChernClasses}). For a $d$-dimensional smooth subvariety $\mathcal X$ of a tropical abelian variety, its tropical CSM classes satisfy the signed positivity property $(-1)^{d-k}\operatorname{csm}_k(\mathcal X)$ is effective.

\begin{question}
Can one define suitable Chern or CSM classes for H-regular tropical varieties, and do they satisfy the same signed positivity property? More generally, can the signed Euler characteristic inequality of Theorem~\ref{thm:abelian variety} be refined to a signed positivity statement for tropical Chern classes?
\end{question}

\appendix
\section{M2 code for verifying Proposition \ref{ambient}}\label{M2cert}
The following Macaulay2 \cite{M2} code verifies portions of Proposition \ref{ambient}.
It requires roughly 80 seconds with peak memory usage 95MB in Macaulay2 1.22 on an Intel Core i5-1240P (a pretty common laptop processor).

\begin{verbatim}
restart
needsPackage "Polyhedra"

-*
A: linear transform
F: a fan
returns: image of F under A as a fan.
If the images of the cones do not form a fan, an error is thrown per M2 doc
macaulay2.com/doc/Macaulay2-1.22/share/doc/Macaulay2/Polyhedra/html/_fan.html
*-
imgFan = (A, F) -> fan((maxCones F) / (C -> coneFromVData(A * (rays F)_C)));

-*
S: a subset of {1...n}
n: a positive integer
returns: \sum_{i \in S} e_i
*-
std = (S,n) -> transpose matrix {
                 for i from 1 to n list if member(i,S) then 1 else 0
               }


a = std({1,2,3}, 9);
b = std({4,5,6},9);
l = a-b; pr = transpose gens ker transpose l;
Sigma0 = skeleton(3, normalFan simplex 9);
Sigma0 = stellarSubdivision(stellarSubdivision(Sigma0, a), b);

-- no error thrown here means the images of the cones form a fan
Sigma = imgFan(pr, Sigma0);

-- proves that the only two cones of Sigma0 that map to the same cone
-- downstairs are \rho_a and \rho_b
print((for i from 1 to 3 list #cones(i, Sigma0) - #cones(i, Sigma)) == {1,0,0})
\end{verbatim}

\printbibliography

@misc{AP24,
      title={All triangulations have a common stellar subdivision}, 
      author={Karim Adiprasito and Igor Pak},
      year={2024},
      eprint={2404.05930},
      archivePrefix={arXiv},
      primaryClass={math.CO},
      url={https://arxiv.org/abs/2404.05930}, 
}

@article{AW,
    AUTHOR = {Arapura, Donu and Wang, Botong},
     TITLE = {Perverse sheaves on varieties with large fundamental groups},
   JOURNAL = {J. Differential Geom.},
  FJOURNAL = {Journal of Differential Geometry},
    VOLUME = {129},
      YEAR = {2025},
    NUMBER = {1},
     PAGES = {1--15},
      ISSN = {0022-040X,1945-743X},
   MRCLASS = {32Q55 (14D07 14F06 32G20)},
  MRNUMBER = {4856131},
MRREVIEWER = {Lutian\ Zhao},
       DOI = {10.4310/jdg/1736261441},
       URL = {https://doi-org.ezproxy.library.wisc.edu/10.4310/jdg/1736261441},
}

@incollection {BBD,
    AUTHOR = {Be{i}linson, A. A. and Bernstein, J. and Deligne, P.},
     TITLE = {Faisceaux pervers},
 BOOKTITLE = {Analysis and topology on singular spaces, {I} ({L}uminy,
              1981)},
    SERIES = {Ast\'erisque},
    VOLUME = {100},
     PAGES = {5--171},
 PUBLISHER = {Soc. Math. France, Paris},
      YEAR = {1982},
   MRCLASS = {32C38},
  MRNUMBER = {751966},
MRREVIEWER = {Zoghman\ Mebkhout},
}

@article {BW,
    AUTHOR = {Budur, Nero and Wang, Botong},
     TITLE = {The signed {E}uler characteristic of very affine varieties},
   JOURNAL = {Int. Math. Res. Not. IMRN},
  FJOURNAL = {International Mathematics Research Notices. IMRN},
      YEAR = {2015},
    NUMBER = {14},
     PAGES = {5710--5714},
      ISSN = {1073-7928,1687-0247},
   MRCLASS = {14R10},
  MRNUMBER = {3384454},
MRREVIEWER = {Dan\ Yan},
       DOI = {10.1093/imrn/rnu107},
       URL = {https://doi-org.ezproxy.library.wisc.edu/10.1093/imrn/rnu107},
}

@article {CCKR,
    AUTHOR = {Chan, Melody and Clader, Emily and Klivans, Caroline and Ross,
              Dustin},
     TITLE = {Piecewise-exponential functions and {E}hrhart fans},
   JOURNAL = {J. Lond. Math. Soc. (2)},
  FJOURNAL = {Journal of the London Mathematical Society. Second Series},
    VOLUME = {113},
      YEAR = {2026},
    NUMBER = {1},
     PAGES = {Paper No. e70434, 49},
      ISSN = {0024-6107,1469-7750},
   MRCLASS = {14M25 (05E14 14T15 52B20 52B40)},
  MRNUMBER = {5020577},
       DOI = {10.1112/jlms.70434},
       URL = {https://doi-org.ezproxy.library.wisc.edu/10.1112/jlms.70434},
}

@book {CLS11,
    AUTHOR = {Cox, David A. and Little, John B. and Schenck, Henry K.},
     TITLE = {Toric varieties},
    SERIES = {Graduate Studies in Mathematics},
    VOLUME = {124},
 PUBLISHER = {American Mathematical Society, Providence, RI},
      YEAR = {2011},
     PAGES = {xxiv+841},
      ISBN = {978-0-8218-4819-7},
   MRCLASS = {14M25 (05A15 05E45 52B12)},
  MRNUMBER = {2810322},
MRREVIEWER = {Ivan\ Arzhantsev},
       DOI = {10.1090/gsm/124},
       URL = {https://doi-org.ezproxy.library.wisc.edu/10.1090/gsm/124},
}

@article {Deligne,
    AUTHOR = {Deligne, Pierre},
     TITLE = {Th\'eorie de {H}odge. {II}},
   JOURNAL = {Inst. Hautes \'Etudes Sci. Publ. Math.},
  FJOURNAL = {Institut des Hautes \'Etudes Scientifiques. Publications
              Math\'ematiques},
    NUMBER = {40},
      YEAR = {1971},
     PAGES = {5--57},
      ISSN = {0073-8301,1618-1913},
   MRCLASS = {14C30 (14F15)},
  MRNUMBER = {498551},
MRREVIEWER = {J.\ H. M. Steenbrink},
       URL = {http://www.numdam.org/item?id=PMIHES_1971__40__5_0},
}

@misc{DW,
      title={Linear Chern-Hopf-Thurston conjecture}, 
      author={Ya Deng and Botong Wang},
      year={2024},
      eprint={2405.12012},
      archivePrefix={arXiv},
      primaryClass={math.AG},
      url={https://arxiv.org/abs/2405.12012}, 
}

@article {DP,
    AUTHOR = {Dimca, Alexandru and Papadima, Stefan},
     TITLE = {Non-abelian cohomology jump loci from an analytic viewpoint},
   JOURNAL = {Commun. Contemp. Math.},
  FJOURNAL = {Communications in Contemporary Mathematics},
    VOLUME = {16},
      YEAR = {2014},
    NUMBER = {4},
     PAGES = {1350025, 47},
      ISSN = {0219-1997,1793-6683},
   MRCLASS = {55N25 (14B12)},
  MRNUMBER = {3231055},
MRREVIEWER = {Andrzej\ Kozlowski},
       DOI = {10.1142/S0219199713500259},
       URL = {https://doi-org.ezproxy.library.wisc.edu/10.1142/S0219199713500259},
}

@article {FK,
    AUTHOR = {Franecki, Joseph and Kapranov, Mikhail},
     TITLE = {The {G}auss map and a noncompact {R}iemann-{R}och formula for
              constructible sheaves on semiabelian varieties},
   JOURNAL = {Duke Math. J.},
  FJOURNAL = {Duke Mathematical Journal},
    VOLUME = {104},
      YEAR = {2000},
    NUMBER = {1},
     PAGES = {171--180},
      ISSN = {0012-7094,1547-7398},
   MRCLASS = {14C40 (32C38 32S60)},
  MRNUMBER = {1769729},
MRREVIEWER = {Laurent\ Manivel},
       DOI = {10.1215/S0012-7094-00-10417-6},
       URL = {https://doi-org.ezproxy.library.wisc.edu/10.1215/S0012-7094-00-10417-6},
}

@article {Hacking,
    AUTHOR = {Hacking, Paul},
     TITLE = {The homology of tropical varieties},
   JOURNAL = {Collect. Math.},
  FJOURNAL = {Universitat de Barcelona. Collectanea Mathematica},
    VOLUME = {59},
      YEAR = {2008},
    NUMBER = {3},
     PAGES = {263--273},
      ISSN = {0010-0757,2038-4815},
   MRCLASS = {14T05 (14F43 14M25)},
  MRNUMBER = {2452307},
MRREVIEWER = {Joaquim\ Ro\'e},
       DOI = {10.1007/BF03191187},
       URL = {https://doi-org.ezproxy.library.wisc.edu/10.1007/BF03191187},
}

@article {LMW,
    AUTHOR = {Liu, Yongqiang and Maxim, Laurentiu and Wang, Botong},
     TITLE = {Topology of subvarieties of complex semi-abelian varieties},
   JOURNAL = {Int. Math. Res. Not. IMRN},
  FJOURNAL = {International Mathematics Research Notices. IMRN},
      YEAR = {2021},
    NUMBER = {14},
     PAGES = {11169--11208},
      ISSN = {1073-7928,1687-0247},
   MRCLASS = {14F45 (14F06 32S22 32S50)},
  MRNUMBER = {4285747},
MRREVIEWER = {P.\ E.\ Newstead},
       DOI = {10.1093/imrn/rnaa242},
       URL = {https://doi-org.ezproxy.library.wisc.edu/10.1093/imrn/rnaa242},
}

@book {MS,
    AUTHOR = {Maclagan, Diane and Sturmfels, Bernd},
     TITLE = {Introduction to tropical geometry},
    SERIES = {Graduate Studies in Mathematics},
    VOLUME = {161},
 PUBLISHER = {American Mathematical Society, Providence, RI},
      YEAR = {2015},
     PAGES = {xii+363},
      ISBN = {978-0-8218-5198-2},
   MRCLASS = {14T05 (05B35 14M25 15A80 52B70)},
  MRNUMBER = {3287221},
MRREVIEWER = {Patrick\ Popescu-Pampu},
       DOI = {10.1090/gsm/161},
       URL = {https://doi-org.ezproxy.library.wisc.edu/10.1090/gsm/161},
}

@incollection {MZ,
    AUTHOR = {Mikhalkin, Grigory and Zharkov, Ilia},
     TITLE = {Tropical curves, their {J}acobians and theta functions},
 BOOKTITLE = {Curves and abelian varieties},
    SERIES = {Contemp. Math.},
    VOLUME = {465},
     PAGES = {203--230},
 PUBLISHER = {Amer. Math. Soc., Providence, RI},
      YEAR = {2008},
      ISBN = {978-0-8218-4334-5},
   MRCLASS = {14T05 (05C38 14H40 14H42)},
  MRNUMBER = {2457739},
       DOI = {10.1090/conm/465/09104},
       URL = {https://doi-org.ezproxy.library.wisc.edu/10.1090/conm/465/09104},
}

@incollection{Saito2,
    AUTHOR = {Saito, Morihiko},
     TITLE = {Introduction to mixed {H}odge modules},
      NOTE = {Actes du Colloque de Th\'eorie de Hodge (Luminy, 1987)},
   JOURNAL = {Ast\'erisque},
  FJOURNAL = {Ast\'erisque},
    NUMBER = {179-180},
      YEAR = {1989},
     PAGES = {10, 145--162},
      ISSN = {0303-1179,2492-5926},
   MRCLASS = {32S35 (14C05 14C30 32J25)},
  MRNUMBER = {1042805},
}

@article {Saito1,
    AUTHOR = {Saito, Morihiko},
     TITLE = {Mixed {H}odge modules},
   JOURNAL = {Publ. Res. Inst. Math. Sci.},
  FJOURNAL = {Kyoto University. Research Institute for Mathematical
              Sciences. Publications},
    VOLUME = {26},
      YEAR = {1990},
    NUMBER = {2},
     PAGES = {221--333},
      ISSN = {0034-5318,1663-4926},
   MRCLASS = {14D07 (14C30 32J25)},
  MRNUMBER = {1047415},
MRREVIEWER = {Min\ Ho\ Lee},
       DOI = {10.2977/prims/1195171082},
       URL = {https://doi-org.ezproxy.library.wisc.edu/10.2977/prims/1195171082},
}

@incollection{saito3,
    AUTHOR = {Saito, Morihiko},
     TITLE = {On {K}oll\'ar's conjecture},
 BOOKTITLE = {Several complex variables and complex geometry, {P}art 2
              ({S}anta {C}ruz, {CA}, 1989)},
    SERIES = {Proc. Sympos. Pure Math.},
    VOLUME = {52, Part 2},
     PAGES = {509--517},
 PUBLISHER = {Amer. Math. Soc., Providence, RI},
      YEAR = {1991},
      ISBN = {0-8218-1490-7},
   MRCLASS = {14D07 (14C30 14F17 32L20)},
  MRNUMBER = {1128566},
MRREVIEWER = {H\'el\`ene\ Esnault},
       DOI = {10.1090/pspum/052.2/1128566},
       URL = {https://doi-org.ezproxy.library.wisc.edu/10.1090/pspum/052.2/1128566},
}

@article {saito4,
    AUTHOR = {Saito, Morihiko},
     TITLE = {Mixed {H}odge complexes on algebraic varieties},
   JOURNAL = {Math. Ann.},
  FJOURNAL = {Mathematische Annalen},
    VOLUME = {316},
      YEAR = {2000},
    NUMBER = {2},
     PAGES = {283--331},
      ISSN = {0025-5831,1432-1807},
   MRCLASS = {14C30 (18E30)},
  MRNUMBER = {1741272},
       DOI = {10.1007/s002080050014},
       URL = {https://doi-org.ezproxy.library.wisc.edu/10.1007/s002080050014},
}

@article {ST,
    AUTHOR = {Sturmfels, Bernd and Tevelev, Jenia},
     TITLE = {Elimination theory for tropical varieties},
   JOURNAL = {Math. Res. Lett.},
  FJOURNAL = {Mathematical Research Letters},
    VOLUME = {15},
      YEAR = {2008},
    NUMBER = {3},
     PAGES = {543--562},
      ISSN = {1073-2780},
   MRCLASS = {14Q15},
  MRNUMBER = {2407231},
MRREVIEWER = {Joaquim\ Ro\'e},
       DOI = {10.4310/MRL.2008.v15.n3.a14},
       URL = {https://doi-org.ezproxy.library.wisc.edu/10.4310/MRL.2008.v15.n3.a14},
}

@article {EGM,
    AUTHOR = {Elduque, Eva and Geske, Christian and Maxim, Laurentiu},
     TITLE = {On the signed {E}uler characteristic property for subvarieties
              of abelian varieties},
   JOURNAL = {J. Singul.},
  FJOURNAL = {Journal of Singularities},
    VOLUME = {17},
      YEAR = {2018},
     PAGES = {368--387},
      ISSN = {1949-2006},
   MRCLASS = {14K12 (14F05 32S60 55N33 58K05)},
  MRNUMBER = {3862497},
MRREVIEWER = {Pawe\l\ Bor\'owka},
       DOI = {10.5427/jsing.2018.17p},
       URL = {https://doi-org.ezproxy.library.wisc.edu/10.5427/jsing.2018.17p},
}

@misc{GS,
      title={Tropical Poincar\'e bundle, Fourier-Mukai transform, and a generalized Poincar\'e formula}, 
      author={Soham Ghosh and Farbod Shokrieh},
      year={2025},
      eprint={2503.12835},
      archivePrefix={arXiv},
      primaryClass={math.AG},
      url={https://arxiv.org/abs/2503.12835}, 
}

@article {GL,
    AUTHOR = {Green, Mark and Lazarsfeld, Robert},
     TITLE = {Deformation theory, generic vanishing theorems, and some
              conjectures of {E}nriques, {C}atanese and {B}eauville},
   JOURNAL = {Invent. Math.},
  FJOURNAL = {Inventiones Mathematicae},
    VOLUME = {90},
      YEAR = {1987},
    NUMBER = {2},
     PAGES = {389--407},
      ISSN = {0020-9910,1432-1297},
   MRCLASS = {32G10 (14C22 14H15)},
  MRNUMBER = {910207},
MRREVIEWER = {Sheldon\ H.\ Katz},
       DOI = {10.1007/BF01388711},
       URL = {https://doi-org.ezproxy.library.wisc.edu/10.1007/BF01388711},
}

@article {PP,
    AUTHOR = {Pareschi, Giuseppe and Popa, Mihnea},
     TITLE = {G{V}-sheaves, {F}ourier-{M}ukai transform, and generic
              vanishing},
   JOURNAL = {Amer. J. Math.},
  FJOURNAL = {American Journal of Mathematics},
    VOLUME = {133},
      YEAR = {2011},
    NUMBER = {1},
     PAGES = {235--271},
      ISSN = {0002-9327,1080-6377},
   MRCLASS = {14F05 (14F17)},
  MRNUMBER = {2752940},
MRREVIEWER = {Daniele\ Faenzi},
       DOI = {10.1353/ajm.2011.0000},
       URL = {https://doi-org.ezproxy.library.wisc.edu/10.1353/ajm.2011.0000},
}

@misc{ChernClasses,
      title={Chern Classes of Tropical Manifolds}, 
      author={Lucía López de Medrano and Felipe Rincón and Kris Shaw},
      year={2023},
      eprint={2309.00229},
      archivePrefix={arXiv},
      primaryClass={math.CO},
      url={https://arxiv.org/abs/2309.00229}, 
}

@misc{AB21,
      title={Filtered geometric lattices and Lefschetz Section Theorems over the tropical semiring},
      author={Karim Adiprasito and Anders Björner},
      year={2021},
      eprint={1401.7301},
      archivePrefix={arXiv},
      primaryClass={math.AG},
      url={https://arxiv.org/abs/1401.7301},
}

@article {TropicalHomology,
    AUTHOR = {Itenberg, Ilia and Katzarkov, Ludmil and Mikhalkin, Grigory
              and Zharkov, Ilia},
     TITLE = {Tropical homology},
   JOURNAL = {Math. Ann.},
  FJOURNAL = {Mathematische Annalen},
    VOLUME = {374},
      YEAR = {2019},
    NUMBER = {1-2},
     PAGES = {963--1006},
      ISSN = {0025-5831,1432-1807},
   MRCLASS = {14T05},
  MRNUMBER = {3961331},
MRREVIEWER = {Helge\ Ruddat},
       DOI = {10.1007/s00208-018-1685-9},
       URL = {https://doi-org.ezproxy.library.wisc.edu/10.1007/s00208-018-1685-9},
}

@misc{DW2,
      title={$L^2$-vanishing theorem and a conjecture of Koll\'ar}, 
      author={Ya Deng and Botong Wang},
      year={2025},
      eprint={2409.11399},
      archivePrefix={arXiv},
      primaryClass={math.AG},
      url={https://arxiv.org/abs/2409.11399}, 
}

@article {Payne,
    AUTHOR = {Payne, Sam},
     TITLE = {Boundary complexes and weight filtrations},
   JOURNAL = {Michigan Math. J.},
  FJOURNAL = {Michigan Mathematical Journal},
    VOLUME = {62},
      YEAR = {2013},
    NUMBER = {2},
     PAGES = {293--322},
      ISSN = {0026-2285,1945-2365},
   MRCLASS = {14E30 (14B05 14R99 14T05)},
  MRNUMBER = {3079265},
MRREVIEWER = {I.\ Dolgachev},
       DOI = {10.1307/mmj/1370870374},
       URL = {https://doi-org.ezproxy.library.wisc.edu/10.1307/mmj/1370870374},
}

@article{MikhalkinZiegler2008,
  author = {Grigory Mikhalkin and Günter M. Ziegler},
  title = {Connectivity of weighted rank 3 matroids},
  journal = {Oberwolfach Reports},
  volume = {5},
  number = {4},
  pages = {2477--2556},
  year = {2008},
  note = {Problem Session edited by J. Pach}
}

@article {ADH,
    AUTHOR = {Ardila, Federico and Denham, Graham and Huh, June},
     TITLE = {Lagrangian geometry of matroids},
   JOURNAL = {J. Amer. Math. Soc.},
  FJOURNAL = {Journal of the American Mathematical Society},
    VOLUME = {36},
      YEAR = {2023},
    NUMBER = {3},
     PAGES = {727--794},
      ISSN = {0894-0347,1088-6834},
   MRCLASS = {05B35 (14T15 14T20)},
  MRNUMBER = {4583774},
MRREVIEWER = {Christopher\ Eur},
       DOI = {10.1090/jams/1009},
       URL = {https://doi-org.ezproxy.library.wisc.edu/10.1090/jams/1009},
}

@misc{AHL,
      title={A decomposition theorem for Lefschetz modules}, 
      author={Omid Amini and June Huh and Matt Larson},
      year={2025},
      eprint={2511.02026},
      archivePrefix={arXiv},
      primaryClass={math.AG},
      url={https://arxiv.org/abs/2511.02026}, 
}

@Misc{M2,
author = {Grayson, Daniel R. and Stillman, Michael E.},
title = {Macaulay2, a software system for research in algebraic geometry},
howpublished = {Available at \url{http://www2.macaulay2.com}}
}

@article{AK06,
  author  = {Ardila, Federico and Klivans, Caroline J.},
  title   = {The Bergman Complex of a Matroid and Phylogenetic Trees},
  journal = {Journal of Combinatorial Theory, Series B},
  volume  = {96},
  number  = {1},
  pages   = {38--49},
  year    = {2006},
  doi     = {10.1016/j.jctb.2005.06.004}
}
\end{document}